\documentclass[11pt]{article}
\usepackage[latin1]{inputenc}
\usepackage{amssymb, amsthm, amsmath}
\usepackage{amsfonts}
\usepackage{mathrsfs}
\usepackage{graphicx}
\usepackage{latexsym}
\usepackage{comment}
\usepackage{setspace}
\usepackage[left=2.50cm, right=2.50cm, top=3.00cm, bottom=3.00cm]{geometry}
\newtheorem{definition}{\bf Definition}[section]
\newtheorem{theorem}[definition]{\bf Theorem}
\newtheorem{lemma}{\bf Lemma}[section]

\newtheorem{remark}[definition]{\bf Remark}

\allowdisplaybreaks[0]

\makeatletter

\@addtoreset{equation}{section}
\makeatother
\begin{document}
\title{\bf On some parabolic systems arising from a nuclear reactor model with nonlinear boundary conditions\\[3mm]}
\author{{\large Kosuke Kita}\\[3mm]
	Graduate School of Advanced Science and Engineering, \\ Waseda University, 3-4-1 Okubo Shinjuku-ku, Tokyo, 169-8555, JAPAN\\[5mm]
		{\large Mitsuharu \^{O}tani}\\[3mm]
	Department of Applied Physics, School of Science and Engineering, \\ Waseda University, 3-4-1 Okubo Shinjuku-ku, Tokyo, 169-8555, JAPAN\\[5mm]
		{\large Hiroki Sakamoto}\\[3mm]
	Hitachi-GE Nuclear Energy, Ltd.\\ 3-1-1, Saiwai-cho, Hitachi-shi, Ibaraki-ken, 317-0073, JAPAN\\[5mm]}
\date{}
%
%
%
\footnotetext[1]{
2010 {\it Mathematics Subject Classification.} Primary: %
35K57; 
Secondary: %
35K61, 
35Q79.
\\
Keywords: initial-boundary problem, reaction diffusion system, nonlinear boundary conditions, threshold property.
}
\footnotetext[2]{Partly supported by the Grant-in-Aid 
	for Scientific Research, \# 15K13451,
	the Ministry of Education, Culture, Sports, Science and Technology, Japan.}
\footnotetext[3]{e-mail : kou5619@asagi.waseda.jp}
\maketitle
\noindent
{\bf Abstract.}
In this paper, we are concerned with a reaction diffusion system arising from a nuclear reactor model in bounded domains with nonlinear boundary conditions.  
We show the existence of a stationary solution and its ordered uniqueness. It is also shown that every positive stationary solution possesses threshold property to determine blow-up or globally existence for solutions of nonstationary problem. \\

\newpage

\section{Introduction}
\hspace{\parindent}We consider the following initial-boundary value problem for a nonlinear reaction diffusion system:
\begin{equation}
	\tag{NR}
	\label{NR}
	\left\{
	\begin{aligned}
		& \partial_t u_1-\Delta u_1 = u_1u_2 - bu_1, && x\in\Omega,~t>0,\\
		& \partial_t u_2-\Delta u_2 = au_1,  && x\in\Omega,~t>0,\\
		& \partial_{\nu}u_1 + \alpha u_1 = \partial_{\nu}u_2 + \beta|u_2|^{\gamma-2}u_2 = 0, && x\in\partial\Omega,~t>0,\\
		& u_1(x,0)=u_{10}(x)\ge0,~u_2(x,0)=u_{20}(x)\ge0, &&x\in\Omega,
	\end{aligned}
	\right.
\end{equation}
where \(\Omega\subset\mathbb{R}^N\) is a bounded domain with smooth boundary \(\partial\Omega\), 
\(\nu\) denotes the unit outward normal vector on \(\partial\Omega\) and \(\partial_\nu\) is outward normal derivative, i.e., \(\partial_\nu u_i = \nabla u_i\cdot \nu \) (\(i=1,2\)).
Moreover \(u_1\), \(u_2\) are real-valued unknown functions, \(a\) and \(b\) are given positive constants. As for the parameters appearing in the boundary condition, we assume \(\alpha\in[0,\infty)\), \(\beta\in(0,\infty)\) and \(\gamma\in[2,\infty)\). We note that the boundary condition for \(u_1\) becomes the homogeneous Neumann boundary condition when \(\alpha=0\), and the boundary condition for \(u_2\) gives the Robin boundary condition when \(\gamma=2\). Finally, \(u_{10}\), \(u_{20}\in L^{\infty}(\Omega)\) are 
given nonnegative initial data.

This system describes diffusion phenomena of neutrons and heat in nuclear reactors by taking the heat conduction into consideration, introduced by Kastenberg and Chambr\'{e} \(\cite{KC1}\). In this model \(u_1\) and \(u_2\) represent the neutron density and the temperature in nuclear reactors respectively. There are many studies on this model under various boundary conditions, for example, \(\cite{C1}\), \(\cite{C2}\), \(\cite{GW1}\), \(\cite{GW2}\), \(\cite{JLZ}\), \(\cite{R}\) and \(\cite{Y}\). Many of them are concerned with the existence of positive steady-state solutions and the long-time behavior of solutions.

The original problem for \(\eqref{NR}\):
\begin{equation}
\label{Kastenberg}
\left\{
\begin{aligned}
& \partial_t u_1-\Delta u_1 = u_1u_2 - bu_1, && x\in\Omega,~t>0,\\
& \partial_t u_2 = au_1 - cu_2,  && x\in\Omega,~t>0,\\
& u_1 = 0, && x\in\partial\Omega,~t>0,\\
& u_1(x,0)=u_{10}(x)\ge0,~u_2(x,0)=u_{20}(x)\ge0, &&x\in\Omega,
\end{aligned}
\right.
\end{equation}
for some \(c>0\) is studied by \(\cite{R}\). In \(\eqref{Kastenberg}\), the negative feedback \(-cu_2\) from the heat into itself is considered instead of the diffusion term. In Rothe's book \(\cite{R}\), the boundedness and the convergence to equilibrium for \(\eqref{Kastenberg}\) are examined in detail.

In \(\cite{GW1}\), our system is studied with \(\alpha=0\) and \(\gamma=2\), i.e., with the homogeneous Neumann boundary condition and Robin boundary condition:
\begin{equation}
	\label{1.2}
	\left\{
	\begin{aligned}
		& \partial_t u_1-\Delta u_1 = u_1u_2 - bu_1, && x\in\Omega,~t>0,\\
		& \partial_t u_2-\Delta u_2 = au_1,  && x\in\Omega,~t>0,\\
		& \partial_{\nu}u_1= \partial_{\nu}u_2 + \beta u_2 = 0, && x\in\partial\Omega,~t>0,\\
		& u_1(x,0)=u_{10}(x),~u_2(x,0)=u_{20}(x), &&x\in\overline\Omega.
	\end{aligned}
	\right.
\end{equation}
They showed the existence and the ordered uniqueness of positive stationary solution for \(N\in[2,5]\). They also investigated some threshold property to determine blow-up or globally existence. Moreover, in \(\cite{Y}\) the case where \(\beta=0\), that is, the homogeneous Neumann boundary condition for \(u_2\) is studied. The author of \(\cite{Y}\) discussed the stability region and the instability region of \(\eqref{1.2}\) and give an upper bound and a lower bound on the blowing-up time for a solution which blows up in finite time.

The following system with the homogeneous Dirichlet boundary conditions:
\begin{equation}
	\label{1.3}
	\left\{
	\begin{aligned}
		& \partial_t u_1-\Delta u_1 = u_1u_2^p - bu_1, && x\in\Omega,~t>0,\\
		& \partial_t u_2-\Delta u_2 = au_1,  && x\in\Omega,~t>0,\\
		& u_1= u_2 = 0, && x\in\partial\Omega,~t>0,\\
		& u_1(x,0)=u_{10}(x),~u_2(x,0)=u_{20}(x), &&x\in\Omega,
	\end{aligned}
	\right.
\end{equation}
is studied by \(\cite{GW2}\) and \(\cite{JLZ}\). In \(\cite{GW2}\), they showed the existence of positive stationary solutions for the case where \(p=1\) and \(N=2\), \(3\) or \(\Omega\) is bounded convex domain with \(N\in[2,5]\). Furthermore, they obtained the threshold property of stationary solution announced in \(\cite{GW1}\) when \(\Omega\) is ball. In \(\cite{JLZ}\), the existence and ordered uniqueness of positive stationary solutions are considered  for general \(p>0\) and some threshold result is obtained. Moreover the blow-up rate estimate is given for positive blowing-up solutions when \(\Omega\) is ball and \(p\ge1\). 

In this paper, we are concerned with the nonlinear boundary condition. From physical point of view it could be more natural to consider the nonlinear boundary condition rather than the homogeneous Dirichlet boundary condition or Neumann boundary condition. Indeed, if there is no control of the heat flux on the boundary, it is well known that the power type nonlinearity for \(u_2\) is justified by Stefan-Boltzmann's law, which says that the heat energy radiation from the surface of the body is proportional to the fourth power of temperature when \(N=3\).  

The outline of this paper is as follows.
In Section 2, we consider the stationary problem associated with  \(\eqref{NR}\) and show the existence of positive solutions by applying an abstract fixed point theorem based on Krasnosel'skii \(\cite{Kr1}\). 
In order to apply this fixed point theorem, we need to estimate \(L^{\infty}\)-norm of solutions. 
To do this, since we are concerned with nonlinear boundary conditions, we can not rely on the standard linear theory.
To cope with this difficulty, we introduce a new approach, which enables us to obtain strong summability of solutions on the boundary.
Next, we prove the ordered uniqueness for the positive stationary solutions of \(\eqref{NR}\). 
We here use the property of first eigenfunction for the eigenvalue problem associated with the Robin boundary condition.

In Section 3, we study the nonstationary problem. 
In the first subsection, we show the existence of local solutions in time for \(\eqref{NR}\) by abstract theory of maximal monotone operators associated with subdifferential operators together with \(L^{\infty}\)-energy method \(\cite{O1}\). 
In the second subsection, we discuss the large time behavior of solutions to \(\eqref{NR}\) and prove that every positive stationary solution plays a role of threshold to separate global solutions and finite time blowing-up solutions. 
In this procedure, we essentially rely on the comparison theorem. 
Furthermore in order to show the finite time blow-up of solutions of \(\eqref{NR}\), the crucial point is to construct an appropriate subsolution.

\section{Stationary problem}
\hspace{\parindent}In this section, we are going to show the existence of the positive stationary solutions for \(\eqref{NR}\) and prove the ordered uniqueness of them. The stationary problem for \(\eqref{NR}\) is given by
\begin{equation}
	\tag{S-NR}
	\label{SNR}
	\left\{
	\begin{aligned}
		& -\Delta u_1 = u_1u_2 - bu_1, && x\in\Omega,\\
		& -\Delta u_2 = au_1,  && x\in\Omega,\\
		& \partial_{\nu}u_1 + \alpha u_1 = \partial_{\nu}u_2 + \beta|u_2|^{\gamma-2}u_2 = 0, && x\in\partial\Omega.
	\end{aligned}
	\right.
\end{equation}
It should be noticed that since \(\eqref{SNR}\) has no variational structure, it is not possible to apply the variational method to \(\eqref{SNR}\). In order to show the existence of positive stationary solutions to \(\eqref{NR}\), we rely on the abstract fixed point theorem developed by Krasnosell'skii. The crucial step in proving the existence of positive stationary solutions is how to obtain \(L^{\infty}\)-estimates of solutions.

We state a couple of lemmas to prove our results for \(\eqref{SNR}\).

\begin{lemma}[Krasnosel'skii-type fixed point theorem \(\cite{Kr1}\), \(\cite{K1}\)]\label{KFP} Suppose that \(E\) is a real Banach space with norm \(\|\cdot\|\), \(K\subset E\) is a positive cone, and \(\varPhi:K\rightarrow K\) is a compact mapping satisfying \(\varPhi(0)=0\). Assume that there exists two constants \(R>r>0\) and an element \(\varphi\in K\setminus \{0\}\), such that
	\begin{center}
		{\rm (i)} \(u\neq\lambda\varPhi(u)\),~~\(\forall\lambda\in(0,1)\),~~if~~\(u\in K\) and \(\|u\|=r\),
		\\[2mm]
		{\rm (ii)} \(u\neq\varPhi(u)+\lambda\varphi\),~~\(\forall\lambda\ge0\),~~if~~\(u\in K\) and \(\|u\|=R\).
	\end{center}
	Then the mapping \(\varPhi\) possesses at least one fixed point in \(K_1:=\{u\in K;~0<r<\|u\|<R\}\).
\end{lemma}

\begin{lemma}[\(\cite{EOE}\)]
	\label{phi1}
	Let \(\lambda_1\) and \(\varphi_1\) be the first eigenvalue and the corresponding eigenfunction for the problem:
	\begin{equation*}
		\left\{
		\begin{aligned}
			& -\Delta\varphi = \lambda\varphi, && x\in\Omega,\\
			& \partial_{\nu}\varphi + \alpha\varphi = 0, &&x\in\partial\Omega,
		\end{aligned}
		\right.
	\end{equation*}
	where \(\Omega\) is smooth bounded domain in \(\mathbb{R}^N\) and \(\alpha>0\). Then \(\lambda_1>0\) and there exists a constant \(C_{\alpha}>0\) such that
	\begin{equation*}
		\varphi_1(x)\ge C_{\alpha}\hspace{10mm}x\in\overline{\Omega}.
	\end{equation*}
\end{lemma}

Indeed, it is well known that \(\varphi_1>0\) in \(\Omega\) by the strong maximum principle. Suppose that there exists \(x_0\in\partial\Omega\) such that \(\varphi_1(x_0)=0\), then the boundary condition assures \(\partial_\nu\varphi_1(x_0)=-\alpha\varphi_1(x_0)=0\). On the other hand, Hopf's strong maximum principle assures that \(\partial_\nu\varphi_1(x_0)<0\). This is contradiction, i.e., \(\varphi_1(x)>0\) on \(\overline{\Omega}\).

\subsection{Existence of positive solutions}
\begin{theorem}
	\label{EoPSS}
	Let \(1\le N\le5\) and suppose that either
	{\rm (A)} or {\rm (B)} is satisfied {\rm :}
	\begin{equation*}
		\left\{
		\begin{aligned}
			& \mbox{\rm (A)}~~~~~\gamma=2,~~\alpha\le2\beta,\\
			& \mbox{\rm (B)}~~~~~\gamma>2.\\
		\end{aligned}
		\right.
	\end{equation*}
	Then \(\eqref{SNR}\) has at least one positive solution.
\end{theorem}

We rely on Lemma \ref{KFP} to prove this theorem. In order to apply Lemma \ref{KFP}, we here fix our setting:
\begin{align*}
	&E=C(\overline{\Omega})\times C(\overline{\Omega}),&&u=(u_1,u_2)^{\mathrm{T}}\in E,\\
	&\|u\|=\|u_1\|_{C(\overline{\Omega})}+\|u_2\|_{C(\overline{\Omega})},&&K=\{u\in E;u_1\ge0,u_2\ge0\}.
\end{align*}
Set \(\varphi=(\varphi_1,0)^{\mathrm{T}}\in K\setminus \{0\}\), where \(\lambda_1\) and \(\varphi_1\) are the first eigenvalue and the corresponding eigenfunction of the following eigenvalue problem:
\begin{equation}
	\label{eigen}
	\left\{
	\begin{aligned}
		& -\Delta\varphi = \lambda\varphi, && x\in\Omega,\\
		& \partial_{\nu}\varphi + \alpha\varphi = 0, &&x\in\partial\Omega.
	\end{aligned}
	\right.
\end{equation}
In section 2, we normalize \(\varphi_1(x)\) such that \(\|\varphi_1\|_{L^2}=1\).
For given \(u=(u_1,u_2)^{\mathrm{T}}\in K\), let \(v=(v_1,v_2)^{\mathrm{T}}=\varPsi(u)\) be the unique nonnegative solution (see Br\'ezis \(\cite{B1}\)) of 
\begin{equation}
	\label{fixed}
	\left\{
	\begin{aligned}
		& -\Delta v_1 + bv_1= u_1u_2, && x\in\Omega,\\
		& -\Delta v_2 = au_1,  && x\in\Omega,\\
		& \partial_{\nu}v_1 + \alpha v_1 = \partial_{\nu}v_2 + \beta|v_2|^{\gamma-2}v_2 = 0, && x\in\partial\Omega.
	\end{aligned}
	\right.
\end{equation}
It is clear that \(\varPsi(0)=0\). Moreover \(\varPsi:K\rightarrow K\) is compact. In order to prove the compactness of \(\varPsi\), we use the next Lemma for the following problem:
\begin{equation}
\label{NBC}
\left\{
\begin{aligned}
&-\Delta u = f,&&x\in \Omega,\\
&\partial_\nu u = g, &&x\in\partial\Omega.
\end{aligned}
\right.
\end{equation}
\begin{lemma}
	\label{regularity}
	{\upshape(\(\cite{N1}\))}
	Let \(\Omega\subset\mathbb{R}^N\) be a bounded Lipschitz domain. Suppose that \(f\in L^{\frac{p}{2}}(\Omega)\) and \(g\in L^{p-1}(\partial\Omega)\) with \(p>N\ge 2\), then there exist \(\delta>0\) and a positive constant \(C\) such that every weak solution \(u\) of \(\eqref{regularity}\) belongs to \(C^{0,\delta}(\overline{\Omega})\) and satisfies
	\begin{equation*}
	\|u\|_{C^{0,\delta}(\overline{\Omega})} \le C \left( \|u\|_{L^2(\Omega)} + \|f\|_{L^{\frac{p}{2}}(\Omega)} + \|g\|_{L^{p-1}(\partial\Omega)} \right).
	\end{equation*}
\end{lemma}
Since \(\Omega\) is bounded and \((u_1,u_2)\in C(\overline{\Omega})\times C(\overline{\Omega})\), it follows from elliptic estimate that \(v_1\in W^{2,p}(\Omega)\) for any \(p\). Since \(W^{2,p}(\Omega)\) is compactly embedded in \( C(\overline{\Omega})\) for \(p>\frac{N}{2}\), the mapping \((u_1,u_2)\mapsto v_1\) is compact. 
Next we assume that \(N\ge2\) and consider the following equation:
\begin{equation*}
\left\{
\begin{aligned}
& -\Delta v_2 = au_1 \in L^{\infty}(\Omega),  && x\in\Omega,\\
& \partial_{\nu}v_2 + \beta|v_2|^{\gamma-2}v_2 = 0, && x\in\partial\Omega.
\end{aligned}
\right.
\end{equation*}
Multiplying the equation by \(|v_2|^{r-2}v_2\) and applying integration by parts, we get
\begin{align}
\label{A}
(r-1)\int_{\Omega}|v_2|^{r-2}|\nabla v_2|^2 dx + \beta \int_{\partial\Omega} |v_2|^{r+\gamma -2}dS = a\int_{\Omega} u_1 |v_2|^{r-2}v_2 dx.
\end{align}
Noting that \((\|\nabla v_2\|_{L^2(\Omega)}^2 + \int_{\partial\Omega} v_2^2 dS)^{1/2}\) is equivalent to the usual \(H^1\)-norm by Poincar\'e-Friedrichs type inequality, we obtain
\begin{align*}
(l.h.s.) &= (r-1)\int_{\Omega} \left| |v_2|^{\frac{r-2}{2}} |\nabla v_2| \right|^2 dx + \beta \int_{\partial\Omega} |v_2|^{r+\gamma -2}dS\\[1.5mm]
&\ge \frac{4(r-1)}{r^2}\int_{\Omega} \left| \nabla|v_2|^{\frac{r}{2}} \right|^2 dx + \beta \int_{\partial\Omega} |v_2|^r dS - \beta|\partial\Omega|\\[1.5mm]
&\ge C_r \left( \int_{\Omega} \left| \nabla|v_2|^{\frac{r}{2}} \right|^2 dx + \int_{\partial\Omega} \left| |v_2|^{\frac{r}{2}} \right|^2 dS \right) -\beta|\partial\Omega|\\[1.5mm]
&\ge  C_r \int_{\Omega} \left| |v_2|^{\frac{r}{2}} \right|^2 dx -\beta|\partial\Omega| =C_r\|v_2\|_{L^r(\Omega)}^r - \beta|\partial\Omega|,
\end{align*}
where \(C_r=\min\{\frac{4(r-1)}{r^2},\beta\}>0\) and we used the estimate: 
\begin{align*}
\beta\int_{\partial\Omega}|v_2|^{r+\gamma-2}dS \ge \beta\int_{\{|v_2|\ge 1\}}|v_2|^{r+\gamma-2}dS &\ge \beta\int_{\{|v_2|\ge 1\}}|v_2|^{r}dS\\[1.5mm]
&= \beta\int_{\partial\Omega}|v_2|^r dS -\beta\int_{\{|v_2|\le 1\}}|v_2|^r dS\\[1.5mm]
&\ge\beta\int_{\partial\Omega}|v_2|^r dS -\beta|\partial\Omega|.
\end{align*}
Hence H\"older's inequality, Young's inequality and \(\eqref{A}\) yield
\begin{equation*}
\|v_2\|_{L^r(\Omega)}\le \left\{ \beta|\partial\Omega|\left( \frac{C_r}{2} \right)^{-1} + \frac{1}{r}\left( \frac{C_r}{2} \right)^{-r}\|au_1\|_{L^r(\Omega)}^r \right\}^{\frac{1}{r}}\hspace{9mm}\forall r<\infty.
\end{equation*}
Therefore by \(\eqref{A}\) we have 
\begin{equation*}
\int_{\partial\Omega} |v_2|^{r+\gamma-2} dS \le \frac{1}{\beta}\|au_1\|_{L^r(\Omega)} \left\{ \beta|\partial\Omega|\left( \frac{C_r}{2} \right)^{-1} + \frac{1}{r}\left( \frac{C_r}{2} \right)^{-r}\|au_1\|_{L^r(\Omega)}^r \right\}^{\frac{r-1}{r}}\hspace{3mm}\forall r<\infty.
\end{equation*}
Thus we see that \(v_2\in L^{r}(\partial\Omega)\) for all large \(r<\infty\) and we can apply Lemma \ref{regularity} to get \(v_2\in C^{0,\delta}(\overline{\Omega})\) for some \(\delta>0\). Note that \(C^{0,\delta}(\overline{\Omega})\hookrightarrow C(\overline{\Omega})\) is compact. 
As for the case where \(N=1\), \(\eqref{A}\) with \(r=2\) gives the a priori bound for \(\|v_2\|_{H^1(\Omega)}\).
Since the embedding \(H^1(\Omega)\hookrightarrow C(\overline{\Omega})\) is compact, the compactness of \(\varPsi\) is easily derived.
Thus we see that \(\varPsi:K\rightarrow K\) is compact.

In order to show the existence of positive stationary solutions for \(\eqref{SNR}\), it suffices to prove that \(\varPsi\) has a fixed point in \(K\). Therefore, to prove Theorem \ref{EoPSS} we are going to verify conditions (i) and (ii) of Lemma \ref{KFP}.

We first check condition (i).

\begin{lemma}
	\label{C1}
	Let \(r=\frac{b}{2}\), then \(u\neq\lambda\varPsi(u)\) for any \(\lambda\in(0,1)\) and \(u\in K\) satisfying \(\|u\|=r\). That is, condition {\rm (i)} of {\rm Lemma \ref{KFP}} with \(\varPhi=\varPsi\) holds.
\end{lemma}

\begin{proof}
	We prove the statement by contradiction. Suppose that there exist \(\lambda\in(0,1)\) and  \(u\in K\) with \(\|u\|=r\) ~such that \(u=\lambda\varPsi(u)\), that is, \(u_1\) and \(u_2\) satisfy 
	\begin{equation}
		\label{hy1}
		\left\{
		\begin{aligned}
			& -\Delta u_1 + bu_1= \lambda u_1u_2, && x\in\Omega,\\
			& -\Delta u_2 = \lambda au_1,  && x\in\Omega,\\
			& \partial_{\nu}u_1 + \alpha u_1 = \partial_{\nu}u_2 + \beta\left|\frac{u_2}{\lambda}\right|^{\gamma-2}u_2 = 0, && x\in\partial\Omega.
		\end{aligned}
		\right.
	\end{equation}
	Multiplying the first equation of \(\eqref{hy1}\) by \(u_1\) and using integration by parts, we obtain 
	\begin{align*}
		\|\nabla u_1\|_{L^2(\Omega)}^2 + \alpha\int_{\partial\Omega}u_1^2dS + b\|u_1\|_{L^2(\Omega)}^2 &= \lambda\int_{\Omega}u_1^2u_2dx\\
		&\le\|u_2\|_{L^{\infty}(\Omega)}\|u_1\|_{L^2(\Omega)}^2\\
		&\le\frac{b}{2}\|u_1\|_{L^2(\Omega)}^2,
	\end{align*}
	where we use the fact 
	\begin{equation*}
		\|u_2\|_{L^{\infty}(\Omega)}\le \|u\|=r=\frac{b}{2}.
	\end{equation*}
	Then 
	\begin{equation*}
		\|\nabla u_1\|_{L^2(\Omega)}^2 + \alpha\int_{\partial\Omega}u_1^2dS + \frac{b}{2}\|u_1\|_{L^2(\Omega)}^2 \le 0.
	\end{equation*}
	Hence we have \(u_1=0\). By the second equation of \(\eqref{hy1}\), we see that \(u_2\) satisfies
	\begin{equation*}
		\left\{
		\begin{aligned}
			& -\Delta u_2 = 0,  && x\in\Omega,\\
			&  \partial_{\nu}u_2 + \beta\left|\frac{u_2}{\lambda}\right|^{\gamma-2}u_2 = 0, && x\in\partial\Omega.
		\end{aligned}
		\right.
	\end{equation*}
	Multiplying this equation by \(u_2\) and integration by parts, we obtain
	\begin{equation*}
		\|\nabla u_2\|_{L^2(\Omega)}^2 + \frac{\beta}{|\lambda|^{\gamma-2}}\int_{\partial\Omega} |u_2|^{\gamma}dS = 0,~~ \mbox{i.e.},~~~~
		\|\nabla u_2\|_{L^2(\Omega)}=0,~~~~u_2\left|\right._{\partial\Omega}=0.
	\end{equation*} 
	By the use of Poincar\'e's inequality, we also get \(u_2=0\). Thus \(u_1=u_2=0\). This contradicts the assumption \(\|u\|=\frac{b}{2}>0\).
\end{proof}

In order to verify condition (ii), we here claim the following lemma.

\begin{lemma}
	\label{C2}
	Let \(1\le N\le5\) and suppose that either
	{\rm (A)} or {\rm (B)} is satisfied {\rm :}
	\begin{equation*}
		\left\{
		\begin{aligned}
			& \mbox{\rm (A)}~~~~~\gamma=2,~~\alpha\le2\beta,\\
			& \mbox{\rm (B)}~~~~~\gamma>2.\\
		\end{aligned}
		\right.
	\end{equation*}
	Then there exists a constant \(R(>r=\frac{b}{2})\) such that for any \(\lambda>0\) and any solution \(u\) of \(u=\varPsi(u)+\lambda\varphi\), it holds that
	\begin{equation*}
		\|u\|<R.
	\end{equation*}
	
\end{lemma}

\begin{proof}
	We rewrite \(u=\varPsi(u)+\lambda\varphi\) in terms of each component:
	\begin{equation}
		\label{hy2}
		\left\{
		\begin{aligned}
			& -\Delta u_1 + bu_1= u_1u_2 + \lambda(b+\lambda_1)\varphi_1, && x\in\Omega,\\
			& -\Delta u_2 = au_1,  && x\in\Omega,\\
			& \partial_{\nu}u_1 + \alpha u_1 = \partial_{\nu}u_2 + \beta|u_2|^{\gamma-2}u_2 = 0, && x\in\partial\Omega.
		\end{aligned}
		\right.
	\end{equation}
	In what follows, we denote by \(C\) a general constant which differs from place to place.
	First, we derive \(H^1\)-estimate for \(u_2\). Replacing \(u_1\) in the first equation of \(\eqref{hy2}\) by \(-\frac{1}{a}\Delta u_2\), we get
	\begin{equation}
		\label{u2}
		\left\{
		\begin{aligned}
			&\Delta^2u_2-b\Delta u_2=-u_2\Delta u_2+\lambda a(b+\lambda_1)\varphi_1,&&x\in\Omega\\
			&\partial_\nu u_2 + \beta|u_2|^{\gamma-2}u_2 = \partial_\nu \Delta u_2 + \alpha \Delta u_2 = 0,&&x\in\partial\Omega.
		\end{aligned}
		\right.
	\end{equation}
	Multiplying \(\eqref{u2}\) by \(\varphi_1\), using integration by parts and noting that the boundary conditions \(\partial_\nu\varphi_1 + \alpha\varphi_1 = \partial_\nu u_2 + \beta|u_2|^{\gamma-2}u_2 =0\), we have
	\begin{align*}
		(l.h.s)&=\int_{\Omega}\Delta^2u_2\varphi_1dx-b\int_{\Omega}\Delta u_2\varphi_1dx\\[1.5mm]
		&=-\int_{\Omega}\nabla(\Delta u_2)\cdot\nabla\varphi_1dx + \int_{\partial\Omega}(\partial_\nu\Delta u_2)\varphi_1dS\\[1.5mm]
		&\phantom{=\int} + b\int_{\Omega}\nabla u_2\cdot\nabla\varphi_1dx -b\int_{\partial\Omega}(\partial_\nu u_2)\varphi_1dS\\[1.5mm]
		&=\int_{\Omega}\Delta u_2 \Delta \varphi_1dx - \int_{\partial\Omega} \Delta u_2 (\partial_\nu\varphi_1)dS+ \int_{\partial\Omega}(\partial_\nu\Delta u_2)\varphi_1dS\\[1.5mm] 
		&\phantom{=\int}-b\int_{\Omega}u_2\Delta\varphi_1dx +b\int_{\partial\Omega}u_2(\partial_\nu\varphi_1)dS -b\int_{\partial\Omega}(\partial_\nu u_2)\varphi_1dS\\[1.5mm]
		&=-\lambda_1\int_{\Omega}\Delta u_2\varphi_1dx + \alpha\int_{\partial\Omega} \Delta u_2 \varphi_1dS - \alpha\int_{\partial\Omega} \Delta u_2 \varphi_1dS \\[1.5mm]
		&\phantom{=\int}+b\lambda_1\int_{\Omega}u_2\varphi_1dx -\alpha b\int_{\partial\Omega}u_2\varphi_1dS + \beta b\int_{\partial\Omega}u_2^{\gamma-1}\varphi_1dS \\[1.5mm]
		&=-\lambda_1 \int_{\Omega} u_2\Delta\varphi_1dx +\lambda_1\int_{\partial\Omega}u_2(\partial_\nu\varphi_1)dS - \lambda_1 \int_{\partial\Omega} (\partial_\nu u_2)\varphi_1dS \\[1.5mm]
		&\phantom{=\int}+b\lambda_1\int_{\Omega}u_2\varphi_1dx -\alpha b\int_{\partial\Omega}u_2\varphi_1dS + \beta b\int_{\partial\Omega}u_2^{\gamma-1}\varphi_1dS \\[1.5mm]
		&=\lambda_1(b+\lambda_1) \int_{\Omega}u_2\varphi_1dx + \beta(b+\lambda_1)\int_{\partial\Omega}u_2^{\gamma-1}\varphi_1dS - \alpha(b+\lambda_1)\int_{\partial\Omega}u_2\varphi_1dS,
	\end{align*}
	and
	\begin{align*}
		(r.h.s)&=-\int_{\Omega}u_2\Delta u_2\varphi_1dx +\lambda a(b+\lambda_1)\|\varphi_1\|_{L^2(\Omega)}^2\\[1.5mm]
		&=\int_{\Omega}\nabla u_2\cdot\nabla(u_2\varphi_1)dx - \int_{\partial\Omega}(\partial_\nu u_2)u_2\varphi_1dS +\lambda a(b+\lambda_1)\\[1.5mm]
		&=\int_{\Omega}|\nabla u_2|^2\varphi_1dx + \int_{\Omega} u_2 \nabla u_2\cdot\nabla\varphi_1dx + \beta\int_{\partial\Omega}u_2^{\gamma}\varphi_1dS +\lambda a(b+\lambda_1)\\[1.5mm]
		&=\int_{\Omega}|\nabla u_2|^2\varphi_1dx + \frac{1}{2}\int_{\Omega}\nabla u_2^2\cdot\nabla\varphi_1dx + \beta\int_{\partial\Omega}u_2^{\gamma}\varphi_1dS +\lambda a(b+\lambda_1)\\[1.5mm]
		&=\int_{\Omega}|\nabla u_2|^2\varphi_1dx - \frac{1}{2}\int_{\Omega}u_2^2\Delta\varphi_1dx + \frac{1}{2}\int_{\partial\Omega}u_2^2(\partial_\nu\varphi_1)dS + \beta\int_{\partial\Omega}u_2^{\gamma}\varphi_1dS +\lambda a(b+\lambda_1)\\[1.5mm]
		&=\int_{\Omega}|\nabla u_2|^2\varphi_1dx +\frac{\lambda_1}{2}\int_{\Omega}u_2^2\varphi_1dx + \beta\int_{\partial\Omega}u_2^{\gamma}\varphi_1dS -\frac{\alpha}{2}\int_{\partial\Omega}u_2^2\varphi_1dS +\lambda a(b+\lambda_1).
	\end{align*}
	Therefore the following equality holds.
	\begin{align}
		\label{byparts}
		\lambda_1( b + \lambda_1 )\int_{\Omega }u_2\varphi_1 dx &= \int_{\Omega }|\nabla u_2|^2\varphi_1dx + \frac{\lambda_1}{2}\int_{\Omega }u_2^2\varphi_1 dx + a(b+\lambda_1)\lambda\\
		&~~~~~+\int_{\partial\Omega} \left\lbrace \beta u_2^{\gamma} -\beta\left( b+\lambda_1 \right)u_2^{\gamma-1} -\frac{\alpha}{2}u_2^2 + \alpha\left( b+\lambda_1 \right)u_2 \right\rbrace \varphi_1 dS.\nonumber
	\end{align}
	Since (A) : \(\gamma=2\), \(\alpha\le2\beta\) or (B) : \(\gamma>2\) holds, we get
	\begin{equation*}
		\inf_{u_2\ge0}\left\lbrace \beta u_2^{\gamma} -\beta\left( b+\lambda_1 \right)u_2^{\gamma-1} -\frac{\alpha}{2}u_2^2 + \alpha\left( b+\lambda_1 \right)u_2 \right\rbrace \ge -C > -\infty.
	\end{equation*}
	Moreover, we see that due to the boundedness of \(\varphi_1\) (cf. Lemma \ref{phi1})
	\begin{equation*}
		\lambda_1( b + \lambda_1 )\int_{\Omega }u_2\varphi_1 dx \ge \int_{\Omega }|\nabla u_2|^2\varphi_1dx + \frac{\lambda_1}{2}\int_{\Omega }u_2^2\varphi_1 dx + a(b+\lambda_1)\lambda -C.
	\end{equation*}
	By Schwarz's inequality and Young's inequality, it is easy to see that
	\begin{align*}
		\int_{\Omega }|\nabla u_2|^2\varphi_1dx + \frac{\lambda_1}{2}\int_{\Omega }u_2^2\varphi_1 dx + a(b+\lambda_1)\lambda &\le \lambda_1( b + \lambda_1 )\int_{\Omega }u_2\varphi_1 dx + C\\[1.5mm]
		&\le \lambda_1(b+\lambda_1) \left(\int_{\Omega}u_2^2\varphi_1dx\right)^{\frac{1}{2}} \|\varphi_1\|_{L^1(\Omega)}^{\frac{1}{2}} +C\\[1.5mm]
		&\le \frac{\lambda_1}{4}\int_{\Omega}u_2^2\varphi_1dx +C.
	\end{align*}
	Hence we obtain
	\begin{align}
		\label{esti1}
		\int_{\Omega}|\nabla u_2|^2\varphi_1dx\le C,~~~~~\int_{\Omega}u_2^2\varphi_1dx\le C,~~~~~\lambda\le C,
	\end{align}
	and 
	\begin{align}
		\label{esti2}
		\int_{\Omega}u_2\varphi_1dx\le\left(\int_{\Omega}u_2^2\varphi_1dx\right)^{\frac{1}{2}}\left(\int_{\Omega}\varphi_1dx\right)^{\frac{1}{2}}\le C.
	\end{align}
	Furthermore it follows from Lemma \ref{phi1} and \(\eqref{esti1}\)
	\begin{equation*}
		C_{\alpha} \left( \int_{\Omega}|\nabla u_2|^2 dx + \int_{\Omega} u_2^2 dx \right) \le \int_{\Omega} |\nabla u_2|^2 \varphi_1 dx + \int_{\Omega} u_2^2\varphi_1 dx \le C,
	\end{equation*}
	whence follows
	\begin{equation}
		\label{esti3}
		\|u_2\|_{H^1(\Omega)}\le C.
	\end{equation}
	By \(\eqref{esti2}\) and \(\eqref{byparts}\), we also have
	\begin{equation}
		\label{boundary}
		\int_{\partial\Omega} \left\lbrace \beta u_2^{\gamma} -\beta\left( b+\lambda_1 \right)u_2^{\gamma-1} -\frac{\alpha}{2}u_2^2 + \alpha\left( b+\lambda_1 \right)u_2 \right\rbrace \varphi_1 dS\le C.
	\end{equation}
	Hence we can obtain
	\begin{equation}
		\label{esti4}
		\begin{cases}
			\displaystyle\int_{\partial\Omega}u_2^{\gamma}\varphi_1dS\le C\hspace{10mm}(\gamma>2~~\mbox{or}~~\gamma=2,~\alpha<2\beta),\\[3mm]
			\displaystyle\int_{\partial\Omega}u_2\varphi_1dS\le C \hspace{10mm}(\gamma=2,~\alpha=2\beta).
		\end{cases}
	\end{equation}
	Indeed, if \(\gamma>2\), then by H\"older's inequality and Young's inequality, we get
	\begin{align*}
		\beta\int_{\partial\Omega} u_2^{\gamma}\varphi_1dS + \alpha(b+\lambda_1)\int_{\partial\Omega}u_2\varphi_1dS &\le C + \beta(b+\lambda_1)\int_{\partial\Omega}u_2^{\gamma-1}\varphi_1dS + \frac{\alpha}{2}\int_{\partial\Omega}u_2^2\varphi_1dS\\[1.5mm]
		&\le C + \beta(b+\lambda_1)\left( \int_{\partial\Omega} u_2^{\gamma}\varphi_1dS \right)^{\frac{\gamma-1}{\gamma}}\left( \int_{\partial\Omega}\varphi_1 dS \right)^{\frac{1}{\gamma}}\\[1.5mm]
		&\phantom{=\int}+ \frac{\alpha}{2} \left( \int_{\partial\Omega} u_2^{\gamma}\varphi_1dS \right)^{\frac{2}{\gamma}} \left( \int_{\partial\Omega}\varphi_1 dS \right)^{\frac{\gamma-2}{\gamma}}\\[1.5mm]
		&\le C +  \beta(b+\lambda_1)\|\varphi_1\|_{L^{\infty}(\Omega)}^{\frac{1}{\gamma}} |\partial\Omega|^{\frac{1}{\gamma}} \left( \int_{\partial\Omega} u_2^{\gamma}\varphi_1dS \right)^{\frac{\gamma-1}{\gamma}}\\[1.5mm] 
		&\phantom{=\int}+\frac{\alpha}{2}\|\varphi_1\|_{L^{\infty}(\Omega)}^{\frac{\gamma-2}{\gamma}} |\partial\Omega|^{\frac{\gamma-2}{\gamma}} \left( \int_{\partial\Omega} u_2^{\gamma}\varphi_1dS \right)^{\frac{2}{\gamma}}\\[1.5mm]
		&\le C + \frac{\beta}{2} \int_{\partial\Omega} u_2^{\gamma}\varphi_1dS,
	\end{align*}
	where we denote by \(|\partial\Omega|\) a volume of \(\partial\Omega\) and use the following property (see \(\cite{HO1}\)):
	\begin{equation*}
		\|\varphi_1\|_{L^{\infty}(\partial\Omega)}\le \|\varphi_1\|_{L^{\infty}(\Omega)}.
	\end{equation*}
	On the other hand, if \(\gamma=2\) and \(\alpha<2\beta\), then it follows from Schwarz's inequality and Young's inequality
	\begin{align*}
		&\left( \beta -\frac{\alpha}{2} \right) \int_{\partial\Omega} u_2^2\varphi_1dS + \alpha(b+\lambda_1)\int_{\partial\Omega}u_2\varphi_1dS\\[1.5mm]
		\le&~ C+ \beta(b+\lambda_1)\int_{\partial\Omega}u_2\varphi_1dS\\[1.5mm]
		\le&~ C+ \beta(b+\lambda_1) \left( \int_{\partial\Omega}u_2^2\varphi_1dS \right)^\frac{1}{2} \left( \int_{\partial\Omega} \varphi_1dS \right)^{\frac{1}{2}}\\[1.5mm]
		\le&~ C+ \beta(b+\lambda_1)\|\varphi_1\|_{L^{\infty}(\Omega)} |\partial\Omega|^{\frac{1}{2}} \left( \int_{\partial\Omega}u_2^2\varphi_1dS \right)^\frac{1}{2}\\[1.5mm]
		\le&~ C+ \frac{1}{2} \left( \beta -\frac{\alpha}{2} \right) \int_{\partial\Omega}u_2^2\varphi_1dS.
	\end{align*}
	For the case of \(\gamma=2\) and \(\alpha=2\beta\), from \(\eqref{boundary}\) it is clear that 
	\begin{equation*}
		\beta\int_{\partial\Omega}u_2\varphi_1dS\le C.
	\end{equation*}
	Thus we obtain \(\eqref{esti4}\).\\
	
	Now, we derive \(H^1\)-estimate for \(u_1\). Multiplying the first equation of \(\eqref{hy2}\) by \(\varphi_1\) and using integration by parts, we get 
	\begin{align}
		\label{3.10}
		(\lambda_1+b)\int_{\Omega}u_1\varphi_1dx=\int_{\Omega}u_1u_2\varphi_1dx+\lambda(\lambda_1+b)
	\end{align}
	Similarly, multiplying the second equation of \(\eqref{hy2}\) by \(\varphi_1\), we get
	\begin{equation}
		\label{3.11}
		\lambda_1\int_{\Omega}u_2\varphi_1dx+\beta\int_{\partial\Omega}u_2^{\gamma-1}\varphi_1dS-\alpha\int_{\partial\Omega}u_2\varphi_1dS=a\int_{\Omega}u_1\varphi_1dx.
	\end{equation} 
	Then by \(\eqref{3.10}\), \(\eqref{3.11}\), \(\eqref{esti3}\) and \(\eqref{esti4}\), we obtain
	\begin{equation}
		\label{3.12}
		\int_{\Omega}u_1\varphi_1dx\le C, ~~~~~\int_{\Omega}u_1u_2\varphi_1dx\le C.
	\end{equation}
	Hence, by Lemma \ref{phi1}, we get a priori bounds for \(\int_{\Omega}u_1dx\) and \(\int_{\Omega}u_1u_2dx\).
	Now we are going to establish a priori bound of \(u_1\) in \(H^1(\Omega)\) for the case of \(N\in[3,5]\). Multiplying the first equation of \(\eqref{hy2}\) by \(u_1\) and using integration by parts, we obtain
	\begin{align}
		\|\nabla u_1\|_{L^2(\Omega)}^2 + \alpha\int_{\partial\Omega}u_1^2 ds + b\|u_1\|_{L^2(\Omega)}^2\nonumber
		&= \int_{\Omega}u_1^2u_2dx + \lambda(b+\lambda_1)\int_{\Omega}u_1\varphi_1dx\\[1.5mm]\nonumber
		&\le \int_{\Omega}\left( u_1u_2 \right)^{\theta} \left( u_1^{\frac{2-\theta}{1-\theta}}u_2 \right)^{1-\theta} dx +C\\[1.5mm]
		&\le \left( \int_{\Omega}u_1u_2 dx\right) ^\theta \left( \int_{\Omega}u_1^{\frac{2-\theta}{1-\theta}}u_2 dx\right) ^{1-\theta} + C,\label{ene1}
	\end{align}
	where we apply H\"older's inequality with exponent \((\frac{1}{\theta},\frac{1}{1-\theta})\) for the first term on the right hand side. Here we take \(\theta=\frac{6-N}{4}\in(0,1)\), then by applying H\"older's inequality with exponent \((\frac{2N}{N+2},\frac{2N}{N-2})\),
	\begin{equation*}
		\left( \int_{\Omega}u_1^{\frac{2-\theta}{1-\theta}}u_2 dx\right) ^{1-\theta} = \left( \int_{\Omega}u_1^{\frac{N+2}{N-2}}u_2 dx\right) ^{\frac{N-2}{4}} \le \|u_1\|_{L^{2^*}(\Omega)}^{\frac{N+2}{4}} \|u_2\|_{L^{2^*}(\Omega)}^{\frac{N-2}{4}}.
	\end{equation*}
	where \(2^*=\frac{2N}{N-2}\) is the critical Sobolev exponent. Using Sobolev's embedding \(H^1(\Omega)\hookrightarrow L^{2^*}(\Omega)\) and \(\eqref{esti3}\), we obtain
	\begin{equation*}
		\|u_1\|_{L^{2^*}(\Omega)}^{\frac{N+2}{4}} \|u_2\|_{L^{2^*}(\Omega)}^{\frac{N-2}{4}} \le  C\|u_1\|_{H^1(\Omega)}^{\frac{N+2}{4}}.
	\end{equation*}
	Since \( (\|\nabla u_1\|_{L^2(\Omega)}^2 + \alpha\int_{\partial\Omega}u_1^2 ds + b\|u_1\|_{L^2(\Omega)}^2)^{1/2} \) is equivalent to the usual \(H^1\)-norm of \(u_1\) due to trace inequality and Poincar\'e-Friedrichs type inequality, as a consequence we have
	\begin{equation*}
		\|u_1\|_{H^1(\Omega)}^2\le C\|u_1\|_{H^1(\Omega)}^{\frac{N+2}{4}} +C.
	\end{equation*}
	Since \(N\in[3,5]\), we have \(\frac{N+2}{4}<2\). Hence it follows from Young's inequality 
	\begin{equation*}
		\|u_1\|_{H^1(\Omega)}^2\le C\|u_1\|_{H^1(\Omega)}^{\frac{N+2}{4}} +C \le \frac{1}{2}\|u_1\|_{H^1(\Omega)}^2 +C.
	\end{equation*}
	Thus we derive
	\begin{equation}
		\label{3.13}
		\|u_1\|_{H^1(\Omega)}\le C.
	\end{equation}
	
	Next, we derive \(L^{\infty}\)-estimates for \(u_1\) as for the case \(N\in[3,5]\). 
	From Sobolev's embedding \(H^1(\Omega)\hookrightarrow L^{\frac{10}{3}}(\Omega)\), we can see that \(u_1\), \(u_2\in L^\frac{10}{3}(\Omega)\) and \(u_1u_2\in L^\frac{5}{3}(\Omega)\). We get \(u_1\in W^{2,\frac{5}{3}}(\Omega)\) by the elliptic estimate for the first equation of \(\eqref{hy2}\). Moreover, \(u_1\in L^5(\Omega)\) by Sobolev's embedding \(W^{2,\frac{5}{3}}(\Omega)\hookrightarrow L^5(\Omega)\). Then by H\"older's inequality,
	\begin{equation*}
	\int_{\Omega} u_1^2u_2^2dx \le \left( \int_{\Omega} u_1^{2\cdot\frac{5}{2}} dx \right)^{\frac{2}{5}} \left( \int_{\Omega} u_2^{2\cdot\frac{5}{3}} \right)^{\frac{3}{5}},
	\end{equation*}
	we can see that \(u_1u_2\in L^2(\Omega)\). By the same reason as before, we know that \(u_1\in W^{2,2}(\Omega)\hookrightarrow L^{10}(\Omega)\). By H\"older's inequality, we have \(u_1u_2\in L^{\frac{5}{2}}(\Omega)\). Hence applying elliptic estimate and Sobolev's embedding again, we get \(u_1\in W^{2,\frac{5}{2}}(\Omega)\hookrightarrow L^q(\Omega)\) for any \(q\in[1,\infty)\). Therefore \(u_1u_2\in L^{\frac{10q}{3q+10}}(\Omega)\) and \(u_1\in W^{2,\frac{10q}{3q+10}}(\Omega)\). Choosing \(q>10\), we have
	\begin{equation*}
	\|u_1\|_{L^{\infty}(\Omega)} \le C_1,
	\end{equation*}
	where we use the Sobolev's embedding \(W^{2,\frac{10q}{3q+10}}(\Omega)\hookrightarrow L^{\infty}(\Omega)\) for \(q>10\).
	
	Thus we obtain \(L^{\infty}\)-estimate of \(u_1\) for the case of \(N\in[3,5]\). About the regularity for \(u_2\), it suffices to consider the following problem for given \(u_1\in L^{\infty}(\Omega)\):
	\begin{equation*}
	\left\{
	\begin{aligned}
	& -\Delta u_2 = au_1 \in L^{\infty}(\Omega),  && x\in\Omega,\\
	& \partial_{\nu}u_2 + \beta|u_2|^{\gamma-2}u_2 = 0, && x\in\partial\Omega.
	\end{aligned}
	\right.
	\end{equation*}
	Therefore we can derive \(L^{\infty}\)-estimate for \(u_2\), i.e.,
	\begin{equation*}
	\|u_2\|_{L^{\infty}(\Omega)}\le C_2
	\end{equation*}
	by the same arguments as for the compactness of \(\varPsi\) applying Lemma \ref{regularity}. Choosing \(R>C_1+C_2\), we can see that the conclusion of this lemma holds.\\
	
	As for the case \(N=1,2\), it suffices to obtain \(L^{\infty}\)-estimate for each component. First, let \(N=2\). Choosing \(\theta=\frac{1}{2}\) in \(\eqref{ene1}\), we see that it follows from Sobolev's embedding \(H^1(\Omega)\hookrightarrow L^p(\Omega)\) ( for all \(p\in[1,\infty)\) )
	\begin{align*}
		\|\nabla u_1\|_{L^2(\Omega)}^2 + \alpha\int_{\partial\Omega}u_1^2 ds + b\|u_1\|_{L^2(\Omega)}^2\nonumber
		&= \int_{\Omega}u_1^2u_2dx + \lambda(b+\lambda_1)\int_{\Omega}u_1\varphi_1dx\\[1.5mm]\nonumber
		&\le \int_{\Omega}\left( u_1u_2 \right)^{\frac{1}{2}} \left( u_1^{3}u_2 \right)^{\frac{1}{2}} dx +C\\[1.5mm]
		&\le \left( \int_{\Omega}u_1u_2 dx\right) ^{\frac{1}{2}} \left( \int_{\Omega}u_1^{3}u_2 dx\right) ^{\frac{1}{2}} + C\\[1.5mm]
		&\le C \left( \int_{\Omega} u_1^3u_2 dx \right)^{\frac{1}{2}} +C\\[1.5mm]
		&\le C \|u_1\|_{L^6(\Omega)}^{\frac{3}{2}}\|u_2\|_{L^2(\Omega)}^\frac{1}{2} +C\\[1.5mm]
		&\le C \|u_1\|_{H^1(\Omega)}^{\frac{3}{2}} + C.
	\end{align*}
	Here we note that we have already had \(H^1\)-estimate for \(u_2\) without restrictions on the space dimension. Thus we also get \(H^1\)-estimate for \(u_1\). In the similar way as for the previous case \(N\in[3,5]\), we can derive \(L^{\infty}\)-estimates for \(u_1\) and \(u_2\).
	
	Let \(N=1\) and \(\Omega=(a_0,b_0)\) with \(a_0<b_0\). Since \(u_1\in C(\overline{\Omega})\), there exists \(x_0\in\overline{\Omega}\) such that 
	\begin{equation*}
		u_1(x_0)=\min_{x\in\overline{\Omega}}u_1(x).
	\end{equation*}
	Furthermore, since it holds that \(\|u_1\|_{L^1(\Omega)}\le C\) for any space dimension, we have
	\begin{equation*}
		\min_{x\in\overline{\Omega}}u_1(x)\le\frac{1}{|\Omega|}\int_{\Omega}u_1dx\le C.
	\end{equation*}
	Here by the fundamental theorem of calculus,
	\begin{equation*}
		u_1(x)=u_1(x_0)+\int_{x_0}^{x}u'_1(\xi) d\xi .
	\end{equation*}
	Therefore we get the following inequality:
	\begin{equation}
		\label{1dim}
		\|u_1\|_{L^{\infty}(\Omega)}\le \int_{a_0}^{b_0}|u'_1(\xi)|d\xi + |u_1(x_0)| \le \|u'_1\|_{L^1(\Omega)} + C.
	\end{equation}
	From \(\eqref{1dim}\), Schwarz's inequality and Young's inequality, we see that
	\begin{align*}
		\|u'_1\|_{L^2}^2 + \alpha\int_{\partial\Omega}u_1^2ds + b\|u_1\|_{L^2}^2 &= \int_{\Omega}u_1^2u_2dx +\lambda(b+\lambda_1)\int_{\Omega}u_1\varphi_1 dx\\[1.5mm]
		&\le \|u_1\|_{L^{\infty}} \int_{\Omega}u_1u_2dx +C\\[1.5mm]
		&\le C\left(  \|u'_1\|_{L^1} + C \right) +C \\[1.5mm]
		&\le C\|u'_1\|_{L^2} + C
		\le \frac{1}{2}\|u'_1\|_{L^2}^2 + C.
	\end{align*}
	Hence we obtain a priori bound for \(\|u_1\|_{H^1(\Omega)}\). Since Sobolev's embedding \(H^1(\Omega)\hookrightarrow L^{\infty}(\Omega)\) holds for \(N=1\), we obtain the desired estimates.
\end{proof}

\begin{proof}[Proof of Theorem \ref{EoPSS}]
	By applying Lemma \ref{C1}, Lemma \ref{C2} and Lemma \ref{KFP}, we can verify that Theorem \ref{EoPSS} holds.
\end{proof}

\subsection{Ordered Uniqueness}
\hspace{\parindent}Next, we discuss the ordered uniqueness of the positive solutions for \(\eqref{SNR}\). We now prepare the following inequality.

\begin{lemma}
	\label{power}
	{\upshape (\(\cite{D}\))}
	For any \(\gamma\in[2,\infty)\), there exists \(C_{\gamma}>0\) such that 
	\begin{equation*}
		\left( x-y \right) \cdot \left( |x|^{\gamma-2}x - |y|^{\gamma-2}y \right)\ge C_{\gamma}|x-y|^{\gamma}
	\end{equation*}
	for all \(x\), \(y\in\mathbb{R}^N\).
\end{lemma}

%

\begin{theorem}
	\label{OU}
	Let \((u_1,u_2)\) and \((v_1,v_2)\) be two positive solutions of \eqref{SNR} satisfying \(u_1\le v_1\) or \(u_2\le v_2\). Then \(u_1\equiv v_1\) and \(u_2\equiv v_2\).
\end{theorem}

\begin{proof}
	Suppose that \(u_1\not\equiv v_1\) or \(u_2\not\equiv v_2\). Without loss of generality, we only have to consider the case where \(u_2\not\equiv v_2\) and \(u_2\le v_2\). In fact, if \(u_1\le v_1\), by the second equation of \(\eqref{SNR}\) we have
	\begin{equation}
		\label{3.14}
		-\Delta (u_2-v_2) = a (u_1-v_1) \le 0.
	\end{equation}
	Multiplying \(\eqref{3.14}\) by \([u_2-v_2]^+:=\max\{u_2-v_2,0\}\) and using integration by parts, we obtain
	\begin{equation}
		\label{3.15}
		\|\nabla [u_2-v_2]^+\|_{L^2(\Omega)}^2 + \beta\int_{\partial\Omega} [u_2-v_2]^+ \left( |u_2|^{\gamma-2}u_2 - |v_2|^{\gamma-2}v_2 \right)dS \le 0.
	\end{equation}
	Note that by Lemma \ref{power}
	\begin{align*}
		\int_{\partial\Omega}[u_2-v_2]^+ \left( |u_2|^{\gamma-2}u_2 - |v_2|^{\gamma-2}v_2 \right)dS &= \int_{\{u_2\ge v_2\}} (u_2-v_2) \left( |u_2|^{\gamma-2}u_2 - |v_2|^{\gamma-2}v_2 \right)dS\\[1.5mm]
		&\ge\int_{\{u_2\ge v_2\}} C_{\gamma} (u_2-v_2)^{\gamma}dS\\[1.5mm]
		&=C_{\gamma}\int_{\partial\Omega}\left( [u_2-v_2]^+ \right)^{\gamma}dS.
	\end{align*}
	By this inequality and \(\eqref{3.15}\), we get
	\begin{equation*}
		\|\nabla [u_2-v_2]^+\|_{L^2(\Omega)}^2 + C_{\gamma}\int_{\partial\Omega}\left( [u_2-v_2]^+ \right)^{\gamma}dS \le 0.
	\end{equation*}
	Therefore we have
	\begin{equation*}
		\nabla [u_2-v_2]^+=0,
	\end{equation*}
	\begin{equation*}
		\left.[u_2-v_2]^+\right|_{\partial\Omega}=0.
	\end{equation*}
	Hence we deduce \([u_2-v_2]^+\equiv 0\), i.e., \(u_2\le v_2\). \\
	
	Next we consider the following eigenvalue problems:
	\begin{equation}
		\label{3.16}
		\left\{
		\begin{aligned}
			&-\Delta w + \left( b-u_2(x) \right)w = \mu' w&&\mbox{in }\Omega,\\[1.5mm]
			&\partial_\nu w + \alpha w = 0 &&\mbox{on }\partial\Omega,
		\end{aligned}
		\right.
	\end{equation}
	and
	\begin{equation}
		\label{3.17}
		\left\{
		\begin{aligned}
			&-\Delta w + \left( b-v_2(x) \right)w = \eta' w&&\mbox{in }\Omega,\\[1.5mm]
			&\partial_\nu w + \alpha w = 0 &&\mbox{on }\partial\Omega.
		\end{aligned}
		\right.
	\end{equation}
	If necessary, we take some nonnegative constant \(L\ge0\) and add both sides of equations of \(\eqref{3.16}\) and \(\eqref{3.17}\) by \(L\), and we can assume \(U(x):=b-u_2(x)+L\ge 1\) and \(V(x):=b-v_2(x)+L\ge 1\). Thus we consider the following problems in stead of \(\eqref{3.16}\) and \(\eqref{3.17}\):
	\begin{equation}
		\label{3.18}
		\left\{
		\begin{aligned}
			&-\Delta w + U(x)w = \mu w&&\mbox{in }\Omega,\\[1.5mm]
			&\partial_\nu w + \alpha w = 0 &&\mbox{on }\partial\Omega,
		\end{aligned}
		\right.
	\end{equation}
	and
	\begin{equation}
		\label{3.19}
		\left\{
		\begin{aligned}
			&-\Delta w + V(x)w = \eta w&&\mbox{in }\Omega,\\[1.5mm]
			&\partial_\nu w + \alpha w = 0 &&\mbox{on }\partial\Omega.
		\end{aligned}
		\right.
	\end{equation}
	By applying the compactness argument for the associate Rayleigh's quotients of \(\eqref{3.18}\) and \(\eqref{3.19}\) , we know that the smallest positive eigenvalues of \(\eqref{3.18}\) and \(\eqref{3.19}\) are attained and we denote them by \(\mu_0\) and \(\eta_0\). Moreover, thanks to \(u_2\not\equiv v_2\) and \(u_2\le v_2\), we see that \(\eta_0<\mu_0\). On the other hand, since \((u_1,u_2)\) and \((v_1,v_2)\) are positive stationary solutions for \(\eqref{SNR}\), \(u_1>0\) and \(v_1>0\) satisfy
	\begin{equation*}
		\left\{
		\begin{aligned}
			&-\Delta u_1 + \left( b-u_2(x)+L \right)u_1 = L u_1&&\mbox{in }\Omega,\\[1.5mm]
			&\partial_\nu u_1 + \alpha u_1 = 0 &&\mbox{on }\partial\Omega,
		\end{aligned}
		\right.
	\end{equation*}
	and
	\begin{equation*}
		\left\{
		\begin{aligned}
			&-\Delta v_1 + \left( b-v_2(x)+L \right)v_1 = L v_1&&\mbox{in }\Omega,\\[1.5mm]
			&\partial_\nu v_1 + \alpha v_1 = 0 &&\mbox{on }\partial\Omega.
		\end{aligned}
		\right.
	\end{equation*}
	By the fact that the eigenvalue corresponding to the positive eigenfunction is the smallest one, we deduce \(\mu_0=L=\eta_0\). This contradicts \(\eta_0<\mu_0\). Thus the proof is completed.
\end{proof}

\section{Nonstationary Problem}
\hspace{\parindent}In this section, we investigate the large time behavior of solutions to \(\eqref{NR}\) and prove that the positive stationary solution plays a role of threshold 
to classify initial data into two groups; namely corresponding solutions of \(\eqref{NR}\) 
blow up in finite time or exist globally.

\subsection{Local Well-posedness}
\hspace{\parindent}First we state the local well-posedness of problem \(\eqref{NR}\).

\begin{theorem}\label{LWP}
	Assume \((u_{10},u_{20})\in L^{\infty}(\Omega)\times L^{\infty}(\Omega)\). Then there exists  \(T>0\) such that \(\eqref{NR}\) possesses a unique solution \((u_1,u_2)\in( L^{\infty}(0,T;L^{\infty}(\Omega))\cap C([0,T);L^2(\Omega)))^2\) satisfying
	\begin{equation}
		\label{4.0}
		\sqrt{t}\partial_tu_1,\sqrt{t}\partial_tu_2,\sqrt{t}\Delta u_1,\sqrt{t}\Delta u_2\in L^2(0,T;L^2(\Omega)).
	\end{equation}
	Furthermore, if the initial data is nonnegative, then the local solution \((u_1,u_2)\) for \(\eqref{NR}\) is nonnegative.
\end{theorem}

In order to prove this theorem, we rely on \(L^{\infty}\)-energy method developed in \(\cite{O1}\). To this end, we prepare some crucial lemmas.

\begin{lemma}
	\label{L2.4}
	{\upshape (\(\cite{O1}\))}
	Let \(\Omega\) be any domain in \(\mathbb{R}^N\) and assume that exists a number \(r_0\ge 1\) and a constant \(C\) independent of \(r\in[r_0,\infty)\) such that 
	\begin{equation*}
		\|u\|_{L^r(\Omega)}\le C\hspace{10mm}\forall r\in[r_0,\infty),
	\end{equation*}
	then \(u\) belongs to \(L^{\infty}(\Omega)\) and the following property holds.
	\begin{equation}
		\label{infty}
		\lim_{r\rightarrow\infty}\|u\|_{L^r(\Omega)}=\|u\|_{L^{\infty}(\Omega)}.
	\end{equation}
	
	Conversely, assume that \(u\in L^{r_0}(\Omega)\cap L^{\infty}(\Omega)\) for some \(r_0\in[1,\infty)\), then \(u\) satisfies \(\eqref{infty}\).
\end{lemma}

\begin{lemma}
	\label{L2.5}
	{\upshape (\(\cite{O1}\))}
	Let \(y(t)\) be a bounded measurable non-negative function on \([0,T]\) and suppose that there exists \(y_0\ge0\) and a monotone non-decreasing function \(m(\cdot): [0,+\infty)\rightarrow[0,+\infty)\) such that
	\begin{equation*}
		y(t)\le y_0+\int_{0}^{t}m(y(s))ds\hspace{10mm}\mbox{a.e.}~t\in(0,T).
	\end{equation*}
	Then there exists a number \(T_0=T_0(y_0,m(\cdot))\in(0,T]\) such that
	\begin{equation*}
		y(t)\le y_0+1\hspace{10mm}\mbox{a.e.}~t\in[0,T_0].
	\end{equation*}
\end{lemma}

\begin{proof}[Proof of Theorem \ref{LWP}]
	(Existence and regularity) We consider the following approximate problem:
	\begin{equation}
		\label{c-NR}
		\left\{
		\begin{aligned}
			& \partial_t u_1-\Delta u_1 = [u_1]_M[u_2]_M - bu_1, && x\in\Omega,~t>0,\\
			& \partial_t u_2-\Delta u_2 = au_1,  && x\in\Omega,~t>0,\\
			& \partial_{\nu}u_1 + \alpha u_1 = \partial_{\nu}u_2 + \beta |u_2|^{\gamma-2}u_2 = 0, && x\in\partial\Omega,~t>0,\\
			& u_1(x,0)=u_{10}(x),~u_2(x,0)=u_{20}(x), &&x\in\Omega,
		\end{aligned}
		\right.
	\end{equation}
	where \(M>0\) is a given constant and the cut-off function \([u]_M\) is defined by
	\begin{equation*}
		[u]_M=
		\begin{cases}
			M, &u\ge M,\\
			u, &|u|\le M,\\
			-M, &u\le-M.
		\end{cases}
	\end{equation*}
	Since \(u\mapsto [u]_M\) is Lipschitz continuous from \(L^2(\Omega)\) into itself, it is well known that \(\eqref{c-NR}\) has a unique global solution \((u_1,u_2)\) satisfying \(\eqref{4.0}\) by applying the abstract theory on maximal monotone operators developed by H. Br\'ezis \(\cite{B1}\).
	
	By multiplying the first equation of \(\eqref{c-NR}\) by \(|u_1|^{r-2}u_1\) and using integration by parts, 
	\begin{equation*}
		\frac{1}{r}\frac{d}{dt}\|u_1(t)\|_{L^r}^r+(r-1)\int_{\Omega}|\nabla u_1|^2u_1^{r-2}dx+\alpha\int_{\partial\Omega}|u_1|^rdS\le\int_{\Omega}|u_1|^{r}|u_2|dx-b\int_{\Omega}|u_1|^rdx.
	\end{equation*}
	Hence
	\begin{equation*}
		\frac{1}{r}\frac{d}{dt}\|u_1(t)\|_{L^r}^r\le \|u_2(t)\|_{L^{\infty}}\|u_1(t)\|_{L^r}^r.
	\end{equation*}
	Divide both sides by \(\|u_1\|_{L^r}^{r-1}\) and integrate with respect to \(t\) on \([0,t]\), then we get
	\begin{equation*}
		\|u_1(t)\|_{L^r}\le\|u_{10}\|_{L^r} + \int_{0}^{t}\|u_1(\tau)\|_{L^r}\|u_2(\tau)\|_{L^{\infty}}d\tau.
	\end{equation*} 
	Letting \(r\) tend to \(\infty\) (Lemma \ref{L2.4}), we derive
	\begin{equation*}
		\|u_1(t)\|_{L^{\infty}}\le\|u_{10}\|_{L^{\infty}} + \int_{0}^{t}\|u_1(\tau)\|_{L^{\infty}}\|u_2(\tau)\|_{L^{\infty}}d\tau.
	\end{equation*}
	Similarly, we can get the following \(L^{\infty}\) estimate for \(u_2\) ;
	\begin{equation*}
		\|u_2(t)\|_{L^{\infty}}\le\|u_{20}\|_{L^{\infty}} + \int_{0}^{t} a\|u_1(\tau)\|_{L^{\infty}}d\tau.
	\end{equation*}
	Therefore setting \(y(t)=\|u_1(t)\|_{L^{\infty}(\Omega)}+\|u_2(t)\|_{L^{\infty}(\Omega)}\), we get
	\begin{equation*}
		y(t)\le y(0)+\int_{0}^{t}\left(y^2(\tau)+ay(\tau)\right)d\tau.
	\end{equation*}		
	Thus applying Lemma \ref{L2.5}, we find that there exists a number \(T>0\) depending only on \(\|u_{10}\|_{L^{\infty}(\Omega)}\) and \(\|u_{20}\|_{L^{\infty}(\Omega)}\) such that
	\begin{equation*}
		y(t)\le y(0) +1 \hspace{10mm} \mbox{a.e.}~ t\in[0,T].
	\end{equation*}
	In other words, we get
	\begin{equation*}
		\|u_1(t)\|_{L^{\infty}(\Omega)}+\|u_2(t)\|_{L^{\infty}(\Omega)}\le \|u_{10}\|_{L^{\infty}(\Omega)}+\|u_{20}\|_{L^{\infty}(\Omega)}+1 \hspace{10mm}\mbox{a.e.}~ t\in[0,T].
	\end{equation*}
	Hence choosing \(M>\|u_{10}\|_{L^{\infty}(\Omega)}+\|u_{20}\|_{L^{\infty}(\Omega)}+1\), we can see that \((u_1,u_2)\) gives a solution for \(\eqref{NR}\) on \([0,T]\) by the definition of the cut-off function \([u]_M\). Note that even though \(\|u_1(t)\|_{L^r}^{r-1}\) attains zero, we can justify this argument by Proposition 1 in \(\cite{MO1}\). To get the regularity estimate of the solution for \(\eqref{NR}\) is standard, so we omit the details.
	
	(Uniqueness) Let \((u_1,u_2)\) and \((v_1,v_2)\) be two solutions to \(\eqref{NR}\) with initial data \((u_{10},u_{20})\) and \((v_{10}, v_{20})\) respectively. We set \(w_1=u_1-v_1\) and \(w_2=u_2-v_2\). From \(\eqref{NR}\), we have
	\begin{equation}
		\label{uni1}
		\partial_t w_1 - \Delta w_1 = w_1u_2 + v_1w_2 -bw_1,
	\end{equation}
	\begin{equation}
		\label{uni2}
		\partial_t w_2 - \Delta w_2 = aw_1,
	\end{equation}
	\begin{equation*}
		\partial_\nu w_1 + \alpha w_1 = \partial_\nu w_2 + \beta \left( |u_2|^{\gamma-2}u_2 - |v_2|^{\gamma-2}v_2 \right) = 0,\hspace{7mm}\mbox{on}~\partial\Omega.
	\end{equation*}
	We multiply \(\eqref{uni1}\) and \(\eqref{uni2}\) by \(w_1\) and \(w_2\) respectively, integrate over \(\Omega\) and use integration by parts. Then we obtain 
	\begin{align*}
		&\frac{1}{2}\frac{d}{dt}\|w_1(t)\|_{L^2(\Omega)}^2 + \|\nabla w_1\|_{L^2(\Omega)}^2 +\alpha \int_{\partial\Omega}w_1^2dS\\[1.5mm] 
		\le&~ \int_{\Omega}w_1^2u_2dx +\int_{\Omega}v_1w_1w_2dx\\[1.5mm]
		\le& ~\|u_2\|_{L^{\infty}(0,T;L^{\infty}(\Omega))}\int_{\Omega}w_1^2dx + \|v_1\|_{L^{\infty}(0,T;L^{\infty}(\Omega))}\int_{\Omega}w_1w_2dx\\[1.5mm]
		\le& ~C \left( \|w_1(t)\|_{L^2(\Omega)}^2 + \|w_2(t)\|_{L^2(\Omega)}^2 \right),
	\end{align*}
	and 
	\begin{align*}
		&\frac{1}{2}\frac{d}{dt}\|w_2(t)\|_{L^2(\Omega)}^2 + \|\nabla w_2\|_{L^2(\Omega)}^2 +\beta \int_{\partial\Omega}\left( |u_2|^{\gamma-2}u_2 - |v_2|^{\gamma-2}v_2 \right) \left( u_2-v_2 \right)dS\\[1.5mm] \le&~ a \int_{\Omega}w_1w_2dx\\[1.5mm]
		\le&~ \frac{a}{2} \left( \|w_1(t)\|_{L^2(\Omega)}^2 + \|w_2(t)\|_{L^2(\Omega)}^2 \right).
	\end{align*}
	Noting that 
	\begin{equation*}
		\int_{\partial\Omega}\left( |u_2|^{\gamma-2}u_2 - |v_2|^{\gamma-2}v_2 \right) \left( u_2-v_2 \right)dS \ge \int_{\partial\Omega}C_{\gamma} |w_2|^{\gamma}dS\ge 0
	\end{equation*}
	by Lemma \ref{power}, we can get the following differential inequality:
	\begin{equation*}
		\frac{d}{dt} \left( \|w_1(t)\|_{L^2(\Omega)}^2 + \|w_2(t)\|_{L^2(\Omega)}^2 \right) \le C \left( \|w_1(t)\|_{L^2(\Omega)}^2 + \|w_2(t)\|_{L^2(\Omega)}^2 \right),
	\end{equation*}
	whence, from Gronwall's inequality,
	\begin{equation*}
		\left( \|w_1(t)\|_{L^2(\Omega)}^2 + \|w_2(t)\|_{L^2(\Omega)}^2 \right) \le  \left( \|u_{10}-v_{10}\|_{L^2(\Omega)}^2 + \|u_{20}-v_{20}\|_{L^2(\Omega)}^2 \right) e^{Ct} \hspace{5mm}t\in[0,T).
	\end{equation*}
	This yields the uniqueness of the solution for \(\eqref{NR}\).
	
	(Nonnegativity) 
	Multiplying the first equation of \(\eqref{NR}\) by \(u_1^-:=\max\{-u_1,0\}\), we get
	\begin{equation*}
		\int_{\Omega}\partial_t u_1u_1^-dx-\int_{\Omega}\Delta u_1u_1^-dx \ge - \int_{\Omega}|u_1^-|^2|u_2|dx-b\int_{\Omega}u_1u_1^-dx.
	\end{equation*}
	Here, we can see that
	\begin{equation*}
		\int_{\Omega}\partial_tu_1 u_1^-dx =\int_{\{u_1\le0\}}\partial_tu_1 (-u_1)dx=-\frac{1}{2}\frac{d}{dt}\int_{\{u_1\le0\}}(-u_1)^2dx=-\frac{1}{2}\frac{d}{dt}\int_{\Omega}\left(u_1^-\right)^2dx,
	\end{equation*}
	and
	\begin{align*}
		-\int_{\Omega}\Delta u_1 u_1^-dx &= \int_{\Omega}\nabla u_1\cdot\nabla u_1^-dx + \alpha \int_{\partial\Omega}u_1u_1^-dS\\[1.5mm]
		&= -\int_{\Omega}|\nabla u_1^-|^2dx - \alpha\int_{\{u_1\le0\}}u_1^2dS = - \int_{\Omega}|\nabla u_1^-|^2dx - \alpha \int_{\partial\Omega}\left(u_1^-\right)^2dS.
	\end{align*}
	Therefore we have
	\begin{align*}
		\frac{1}{2}\frac{d}{dt}\|u_1^-(t)\|_{L^2(\Omega)}^2 + \|\nabla u_1^-\|_{L^2(\Omega)}^2 + \alpha \int_{\partial\Omega}\left( u_1^- \right)^2dS &= \int_{\Omega}|u_1^-|^2|u_2|dx
		 - b\|u_1^-(t)\|_{L^2(\Omega)}^2\\[1.5mm]
		&\le \|u_2\|_{L^{\infty}(0,T;L^{\infty}(\Omega))} \|u_1^-(t)\|_{L^2(\Omega)}^2.
	\end{align*}
	Applying Gronwall's inequality, we obtain
	\begin{equation*}
		\|u_1^-(t)\|_{L^2(\Omega)}^2\le \|u_1^-(0)\|_{L^2(\Omega)}^2 e^{2\|u_2\|_{L^{\infty}(0,T;L^{\infty}(\Omega))}t}\hspace{7mm}t\in[0,T),
	\end{equation*}
	where \(T\) is maximal existence time for \(\eqref{NR}\). Since \(u_{10}\ge 0\), i.e., \(\|u_1^-(0)\|_{L^2(\Omega)}=0\), it holds that 
	\begin{align*}
		u_1^-(t)=0 \hspace{7mm} a.e.~~ \mbox{in}~\Omega\hspace{5mm}\forall t\in[0,T).
	\end{align*}
	Hence \(u_1\ge 0\). Similarly, multiplying the second equation of \(\eqref{NR}\) by \(-u_2^-\), we get
	\begin{equation*}
		\frac{1}{2}\frac{d}{dt}\|u_2^-(t)\|_{L^2(\Omega)}^2 + \|\nabla u_2^-\|_{L^2(\Omega)}^2 + \beta \int_{\partial\Omega}|u_2|^{\gamma-2}|u_2^-|^2 dS = -a \int_{\Omega}u_1u_2^-dx \le 0.
	\end{equation*}
	Therefore \(\|u_2^-(t)\|_{L^2(\Omega)}^2\le \|u_2^-(0)\|_{L^2(\Omega)}^2=0\), i.e., \(u_2\ge 0\). 
\end{proof}

\subsection{Threshold Property}
\hspace{\parindent}Finally, we study the threshold property and prove that every positive stationary solution for \(\eqref{NR}\) gives a threshold for the blow up of solutions in the following sense.

\begin{theorem}
	\label{Thres}
	Let \((\overline{u}_1,\overline{u}_2)\) be a positive stationary solution of \(\eqref{NR}\), 
	then the followings hold.
	\vspace{2mm}\\
	{\rm (1)} Let \(0\le u_{10}(x)\le\overline{u}_1(x)\), \(0\le u_{20}(x)\le\overline{u}_2(x)\), then the solution \((u_1,u_2)\) of \(\eqref{NR}\) exists globally. In addition, if \(0\le u_{10}(x)\le l_1\overline{u}_1(x)\), \(0\le u_{20}(x)\le l_2\overline{u}_2(x)\) for some \(0<l_1<l_2\le1\), then
	\begin{equation*}
		\lim_{t\rightarrow+\infty} (u_1(x,t),u_2(x,t))=(0,0)~~~\text{\rm pointwisely on}\ 
		\overline{\Omega}.
	\end{equation*}
	{\rm (2)} Assume further \(\gamma=2\), \(\alpha\le2\beta\) and let \(u_{10}(x)\ge l_1\overline{u}_1(x)\), \(u_{20}(x)\ge l_2\overline{u}_2(x)\) for some \(l_1>l_2>1\), then the solution \((u_1,u_2)\) of \(\eqref{NR}\) blows up in finite time.
\end{theorem}	

\begin{remark}\upshape
	The second assertion of Theorem \ref{Thres} is also announced in \(\cite{GW1}\)	for the case where \(\alpha=0\) and \(\gamma=2\). However it seems that their proof contains some serious gaps.
\end{remark}

We first prepare the following comparison theorem.

\begin{lemma}[Comparison theorem]
	\label{CP}
	If \((u_{10},u_{20})\), \((v_{10},v_{20})\) are two initial data for \(\eqref{NR}\) satisfying
	\begin{equation*}
		0\le u_{10}\le v_{10},~~~0\le u_{20}\le v_{20}~~~~~\mbox{on}~\overline{\Omega},
	\end{equation*}
	then the corresponding solutions \((u_1,u_2)\), \((v_1,v_2)\) remain in the initial data order in time interval where the solutions exist, i.e., \(u_1(x,t)\le v_1(x,t)\) and \(u_2(x,t)\le v_2(x,t)\) a.e. \(x\in\Omega\) as long as \((u_1,u_2)\) and \((v_1,v_2)\) exist. 
\end{lemma}

\begin{proof}
	Let \(w_1=u_1-v_1\), \(w_2=u_2-v_2\). By \(\eqref{NR}\) we have
	\begin{equation}
		\label{4.2}
		\left\{
		\begin{aligned}
			&\partial_t w_1 - \Delta w_1 = w_1u_2 + v_1w_2 -bw_1,&&x\in\Omega,~t\in(0,T_m),\\
			&\partial_t w_2 - \Delta w_2 = aw_1,&&x\in\Omega,~t\in(0,T_m),\\
			&\partial_\nu w_1 + \alpha w_1 = \partial_\nu w_2 + \beta \left( |u_2|^{\gamma-2}u_2 - |v_2|^{\gamma-2}v_2 \right) = 0,&&x\in\partial\Omega,~t\in(0,T_m),\\
			&w_1(x,0)\le 0,~~ w_2(x,0)\le 0,&&x\in\overline{\Omega},
		\end{aligned}
		\right.
	\end{equation}
	where \(T_m>0\) is the maximum existence time for \((u_1,u_2)\) and \((v_1,v_2)\). We set
	\begin{equation*}
		w^+=w \vee 0, ~~w^- = (-w) \vee 0,
	\end{equation*}
	where \(a\vee b=\max\{a,b\}\). It is easy to see that \(w^+\), \(w^-\ge0\) and 
	\begin{equation*}
		w=w^+-w^-, ~~ |w|=w^+ + w^-.
	\end{equation*}
	Multiplying the first equation of \(\eqref{4.2}\) by \(w_1^+\), we get
	\begin{align*}
		\int_{\Omega}\partial_tw_1 w_1^+dx - \int_{\Omega}\Delta w_1 w_1^+dx = \int_{\Omega}w_1u_2w_1^+dx + \int_{\Omega}v_1w_2w_1^+dx - b\int_{\Omega}w_1w_1^+dx.
	\end{align*}
	Here, we see that
	\begin{equation*}
		\int_{\Omega}\partial_tw_1 w_1^+dx =\int_{\{w_1\ge0\}}\partial_tw_1 w_1dx=\frac{1}{2}\frac{d}{dt}\int_{\{w_1\ge0\}}w_1^2dx=\frac{1}{2}\frac{d}{dt}\int_{\Omega}\left(w_1^+\right)^2dx.
	\end{equation*}
	Similarly,
	\begin{align*}
		-\int_{\Omega}\Delta w_1 w_1^+dx &= \int_{\Omega}\nabla w_1\cdot\nabla w_1^+dx + \alpha \int_{\partial\Omega}w_1w_1^+dS\\[1.5mm]
		&= \int_{\{w_1\ge0\}}|\nabla w_1|^2dx + \alpha\int_{\{w_1\ge0\}}w_1^2dS = \int_{\Omega}|\nabla w_1^+|^2dx + \alpha \int_{\partial\Omega}\left(w_1^+\right)^2dS.
	\end{align*}
	Hence noting that \(v_1\ge 0\), we obtain for any \(T\in(0,T_m)\)
	\begin{align*}
		&\frac{1}{2}\frac{d}{dt}\int_{\Omega}\left(w_1^+\right)^2dx + \int_{\Omega}|\nabla w_1^+|^2dx + \alpha \int_{\partial\Omega}\left(w_1^+\right)^2dS\\[1.5mm] 
		=&~ \int_{\Omega}w_1u_2w_1^+dx + \int_{\Omega}v_1w_2w_1^+dx - b\int_{\Omega}w_1w_1^+dx\\[1.5mm]
		=&~ \int_{\Omega}\left( w_1^+ - w_1^- \right)u_2 w_1^+dx +\int_{\Omega}v_1\left( w_2^+ - w_2^- \right)w_1^+dx -b \int_{\Omega}\left(w_1^+\right)^2dx\\[1.5mm]
		\le&~ \|u_2\|_{L^{\infty}(0,T;L^{\infty}(\Omega))}\int_{\Omega}\left(w_1^+\right)^2dx+ \|v_1\|_{L^{\infty}(0,T;L^{\infty}(\Omega))}\int_{\Omega}w_1^+w_2^+dx\\[1.5mm]
		\le&~ C\left( \|w_1^+(t)\|_{L^2(\Omega)}^2 + \|w_2^+(t)\|_{L^2(\Omega)}^2 \right).
	\end{align*}
	Hence we get
	\begin{equation}
		\label{4.3}
		\frac{1}{2}\frac{d}{dt}\|w_1^+(t)\|_{L^2(\Omega)}^2\le C \left( \|w_1^+(t)\|_{L^2(\Omega)}^2 + \|w_2^+(t)\|_{L^2(\Omega)}^2 \right).
	\end{equation}
	Next we do the same calculation for the second equation of \(\eqref{4.2}\) and get
	\begin{equation*}
		\frac{1}{2}\frac{d}{dt}\int_{\Omega}\left(w_2^+\right)^2dx + \int_{\Omega}|\nabla w_2^+|^2dx - \int_{\partial\Omega}(\partial_\nu w_2) w_2^+dS \le \frac{a}{2}\left( \|w_1^+(t)\|_{L^2(\Omega)}^2 + \|w_2^+(t)\|_{L^2(\Omega)}^2 \right),
	\end{equation*}
	and 
	\begin{align*}
		-\int_{\partial\Omega}(\partial_\nu w_2)w_2^+dS &= \beta \int_{\partial\Omega} \left( |u_2|^{\gamma-2}u_2 - |v_2|^{\gamma-2}v_2 \right) w_2^+dS\\[1.5mm]
		&=\beta \int_{\{u_2\ge v_2\}} \left( |u_2|^{\gamma-2}u_2 - |v_2|^{\gamma-2}v_2 \right) \left( u_2 - v_2 \right) dS \ge 0.
	\end{align*}
	Therefore
	\begin{equation}
		\label{4.4}
		\frac{1}{2}\frac{d}{dt}\|w_2^+(t)\|_{L^2(\Omega)}^2 \le \frac{a}{2} \left( \|w_1^+(t)\|_{L^2(\Omega)}^2 + \|w_2^+(t)\|_{L^2(\Omega)}^2 \right).
	\end{equation}
	Thus by \(\eqref{4.3}\), \(\eqref{4.4}\) and Gronwall's inequality, we get 
	\begin{equation*}
		\|w_1^+(t)\|_{L^2(\Omega)}^2 + \|w_2^+(t)\|_{L^2(\Omega)}^2 \le \left( \|w_1^+(0)\|_{L^2(\Omega)}^2 + \|w_2^+(0)\|_{L^2(\Omega)}^2 \right) e^{Ct}\hspace{5mm}\forall t\in[0,T_m).
	\end{equation*}
	Since \(w_1^+(0)=w_2^+(0)=0\), the above inequality means \(w_1^+=w_2^+=0\). Hence, we have the desired result.
\end{proof}

\begin{proof}[Proof of Theorem \ref{Thres}]
	(1) If \(0\le u_{10}\le\overline{u}_1\) and \(0\le u_{20}\le\overline{u}_2\), then since \((\overline{u}_1,\overline{u}_2)\) is a global solution for \(\eqref{NR}\), \(0\le u_1(x,t)\le\overline{u}_1(x)\) and \(0\le u_2(x,t)\le\overline{u}_2(x)\) follow directly from Lemma \ref{CP}. That is, we have
	\begin{equation*}
		\sup_{t\in[0,T)}\|u_i(\cdot,t)\|_{L^{\infty}(\Omega)} \le \|\overline{u}_i\|_{L^{\infty}(\Omega)}\hspace{5mm}(i=1,2).
	\end{equation*}
	Hence the solution \((u_1,u_2)\) exists globally.
	
	In addition, let \(u_{10}(x)\le l_1\overline{u}_1(x)\), \(u_{20}(x)\le l_2\overline{u}_2(x)\) for some \(0<l_1<l_2\le1\). Since the comparison theorem holds, without loss of generality, we can assume that \(u_{10}(x)=l_1\overline{u}_1(x)\), \(u_{20}(x)=l_2\overline{u}_2(x)\) and \(l_1<l_2\le1\). We here note that \(\delta u_1:=u_1(t+h)-u_1(t)\) and \(\delta u_2:=u_2(t+h)-u_2(t)\) for \(h>0\) satisfy the following equations:
	\begin{equation}
		\label{4.5}
		\left\{
		\begin{aligned}
			&\partial_t\left(\delta u_1\right) - \Delta\left(\delta u_1\right) = \left(\delta u_1\right)u_2(t+h) + u_1(t)\left(\delta u_2\right) - b\left(\delta u_1\right),\\
			&\partial_t\left(\delta u_2\right) - \Delta\left(\delta u_2\right) = a\left(\delta u_1\right),\\
			&\partial_\nu \left(\delta u_1\right) + \alpha\left(\delta u_1\right) = \partial_\nu\left(\delta u_2\right) + \beta \left( |u_2(t+h)|^{\gamma-2}u_2(t+h) - |u_2(t)|^{\gamma-2}u_2(t) \right) = 0,\\
			&\delta u_1(0) = u_1(0+h)-u_1(0),~~\delta u_2(0) = u_2(0+h) - u_2(0).
		\end{aligned}
		\right.
	\end{equation}
	Multiplying the first and second equation of \(\eqref{4.5}\) by \([\delta u_1]^+\) and \([\delta u_2]^+\) respectively and using integration by parts and repeating the same argument as for \(\eqref{4.3}\), we obtain the following inequality:
	\begin{equation*}
		\|[\delta u_1]^+\|_{L^2(\Omega)}^2 + \|[\delta u_2]^+\|_{L^2(\Omega)}^2 \le \left( \|[\delta u_1(0)]^+\|_{L^2(\Omega)}^2 + \|[\delta u_2(0)]^+\|_{L^2(\Omega)}^2 \right) e^{Ct} \hspace{7mm}\forall t\in[0,\infty).
	\end{equation*}
	We divide both sides of this inequality by \(h^2\):
	\begin{equation*}
		\left\| \left[\frac{\delta u_1}{h}\right]^+ \right\|_{L^2(\Omega)}^2 + \left\| \left[\frac{\delta u_2}{h}\right]^+ \right\|_{L^2(\Omega)}^2 \le \left( \left\| \left[\frac{\delta u_1(0)}{h}\right]^+ \right\|_{L^2(\Omega)}^2 + \left\| \left[\frac{\delta u_2(0)}{h}\right]^+ \right\|_{L^2(\Omega)}^2 \right) e^{Ct}.
	\end{equation*}
	Since we know that \(u_1\), \(u_2\) is differentiable on \(a.e.\) \(t\) by the regularity results of Theorem 4.1, by letting \(h\searrow 0\), we obtain
	\begin{equation*}
		\| [\partial_t u_1]^+ \|_{L^2}^2 + \| [\partial_t u_2]^+ \|_{L^2}^2 \le \left( \| [\partial_t u_1(0)]^+ \|_{L^2}^2 + \| [\partial_t u_2(0)]^+ \|_{L^2}^2  \right) e^{Ct}
		\hspace{5mm}a.e.~ t\in[0,\infty).
	\end{equation*}
	We here note that since \((l_1\overline{u}_1, l_2\overline{u}_2)\) is strict upper solution for \(\eqref{SNR}\), it holds that 
	\begin{align*}
		\partial_t u_1(0) &= \Delta u_{10} + u_{10}u_{20} - bu_{10}\\[1.5mm]
		&= l_1\Delta \overline{u}_1 + l_1l_2\overline{u}_1\overline{u}_2 - bl_1 \overline{u}_1\\[1.5mm]
		&\le l_1 \left( \Delta \overline{u}_1 + \overline{u}_1\overline{u}_2 - b \overline{u}_1 \right) =0,\\[1.5mm]
		\partial_t u_2(0) &= \Delta u_{20} + a u_{10}\\[1.5mm]
		&= l_2 \Delta\overline{u}_2 + al_1\overline{u}_1\\[1.5mm]
		&< l_2 \left( \Delta\overline{u}_2 + a\overline{u}_1 \right)=0,
	\end{align*}
	which imply that \([\partial_t u_1(0)]^+=[\partial_t u_2(0)]^+=0\). Hence we find that \(\partial_tu_1\le 0\) and \(\partial_t u_2\le 0\), i.e., \(u_1(x,t)\) and \(u_2(x,t)\) are monotone decreasing in \(t\) for a.e. \(x\in\Omega\). Thus
	\begin{equation*}
		\lim_{t\rightarrow\infty}\left( u_1(x,t), u_2(x,t) \right) =: ( \tilde{u}_1(x), \tilde{u}_2(x) )
	\end{equation*} 
	exists and satisfies \((0,0)\le(\tilde{u}_1, \tilde{u}_2)\le(l_1\overline{u}_1, l_2\overline{u}_2)<(\overline{u}_1,\overline{u}_2)\).
	Now we prove that \((\tilde{u}_1, \tilde{u}_2)\) is a nonnegative stationary solution of \(\eqref{NR}\). 
	First we note that
	\begin{align}
	\label{mod1}
	&u_i(t)\rightarrow \tilde{u}_i \hspace{5mm}strongly\mbox{ ~in }L^p(\Omega)\hspace{5mm}\mbox{as~ }k\rightarrow\infty\hspace{5mm}\forall p\in(1,\infty)\hspace{9mm}(i=1,2).
	\end{align}
	In fact, since \( |u_i(x,t)-\tilde{u}_i(x)|^p\rightarrow 0 \) \(a.e.\) \(x\in\Omega\) as \(t\rightarrow\infty\) and \(|u_i(x,t)-\tilde{u}_i(x)|^p \le2^p|\overline{u}_i(x)|^p \le2^p\|\overline{u}_i\|_{L^{\infty}(\Omega)}^p \) ~\(a.e.\) ~\(x\in\Omega\), Lebesgue's dominant convergence theorem assures \(\eqref{mod1}\).
	Next multiplying the first and the second equations of \(\eqref{NR}\) by \(\partial_t u_1\) and \(\partial_t u_2\) respectively, we get
	\begin{align*}
	&\|\partial_t u_1(t)\|_{L^2(\Omega)}^2 + \frac{d}{dt} \left\{ \frac{1}{2}\|\nabla u_1(t)\|_{L^2(\Omega)}^2 + \frac{\alpha}{2}\|u_1(t)\|_{L^2(\partial\Omega)}^2 + \frac{b}{2}\|u_1(t)\|_{L^2(\Omega)}^2 \right\}\\[1.5mm]
	=& \int_{\Omega} u_1u_2\partial_t u_1dx \le 0,
	\end{align*}
	\begin{align*}
	\|\partial_t u_2(t)\|_{L^2(\Omega)}^2 + \frac{d}{dt} \left\{ \frac{1}{2}\|\nabla u_2(t)\|_{L^2(\Omega)}^2 + \frac{\beta}{\gamma}\|u_2(t)\|_{L^{\gamma}(\partial\Omega)}^\gamma\right\} = a\int_{\Omega} u_1\partial_t u_2dx \le 0.
	\end{align*}
	Then integration of these over \((0,T)\) for any \(T>0\) gives
	\begin{equation}
	\label{mod2}
	\int_{0}^{\infty} \|\partial_t u_1(t)\|_{L^2(\Omega)}^2 dt + \int_{0}^{\infty} \|\partial_t u_2(t)\|_{L^2(\Omega)}^2 dt \le C_0,
	\end{equation}
	\begin{equation}
	\label{mod3}
	\sup_{t>0} \left\{ \|u_1(t)\|_{H^1(\Omega)}^2 + \|u_2(t)\|_{H^1(\Omega)}^2 \right\} \le C_0,
	\end{equation}
	where \(C_0\) is a positive constant depending on \(\|u_{10}\|_{H^1(\Omega)}\), \(\|u_{20}\|_{H^1(\Omega)}\) and \(\|u_{20}\|_{L^{\gamma}(\partial\Omega)}\). Hence since \(u_i\in L^{\infty}(0,\infty;L^{\infty}(\Omega))\) \((i=1,2)\), from equation \(\eqref{NR}\), we derive
	\begin{equation}
	\label{mod4}
	\int_{n}^{n+1} \left\{ \|\partial_t u_1(t)\|_{L^2(\Omega)}^2  + \|\partial_t u_2(t)\|_{L^2(\Omega)}^2 \right\} dt \rightarrow 0 \hspace{9mm}\mbox{as}\hspace{5mm}n\rightarrow\infty,
	\end{equation}
	\begin{equation}
	\label{mod5}
	\sup_{n} \int_{n}^{n+1} \left\{ \|\Delta u_1(t)\|_{L^2(\Omega)}^2 + \|\Delta u_2(t)\|_{L^2(\Omega)}^2 \right\} dt \le C_0.
	\end{equation}
	Furthermore, since \(\|u_2(t)\|_{L^{\infty}(\partial\Omega)} \le \|u_2(t)\|_{L^{\infty}(\Omega)}\) (see \(\cite{HO1}\)), we obtain
	\begin{equation}
	\label{mod6}
	\sup_{t>0} \|u_2(t)\|_{L^{\infty}(\partial\Omega)}\le \|\overline{u}_2\|_{L^{\infty}(\Omega)}.
	\end{equation}
	Here we put 
	\begin{equation}
	\label{mod7}
	u_i^n(x,t) = u_i(x,n+t) \in \mathscr{H} := L^2(0,1;L^2(\Omega)) \hspace{5mm}t\in(0,1)\hspace{9mm}(i=1,2).
	\end{equation}
	Then \(u_i^n(t)\) satisfy
	\begin{equation}
	\left\{
	\begin{aligned}
	&\partial_t u_1^n(t) - \Delta u_1^n(t) = u_1^n(t)u_2^n(t) - b u_1^n(t),&&x\in\Omega,~t\in(0,1),\\[1.5mm]
	&\partial_t u_2^n(t) - \Delta u_2^n(t) = au_1^n(t),&&x\in\Omega,~t\in(0,1),\\[1.5mm]
	&\partial_\nu u_1^n(t) + \alpha u_1^n(t) = \partial_\nu u_2^n(t) + \beta |u_2^n(t)|^{\gamma-2}u_2^n(t) = 0,&&x\in\partial\Omega,~t\in(0,1).
	\end{aligned}
	\right.
	\end{equation}
	Then, by virtue of \(\eqref{mod1}\), \(\eqref{mod3}\), \(\eqref{mod4}\), \(\eqref{mod5}\) and \(\eqref{mod6}\), there exists subsequence of \(\{u_i^n(t)\}\) denoted again by \(\{u_i^n(t)\}\) such that
	\begin{align}
	&\partial_t u_i^n(t) \rightarrow 0&&strongly~\mbox{ in }~\mathscr{H}~\mbox{ as }~n\rightarrow\infty,\\[1.5mm]
	&u_i^n(t)\rightarrow \tilde{u}_i(t)\equiv\tilde{u}_i&&strongly~\mbox{ in }~\mathscr{H}~\mbox{ as }~n\rightarrow\infty,\\[1.5mm]
	&u_1^n(t)u_2^n(t)\rightarrow \tilde{u}_1(t)\tilde{u}_2(t)\equiv\tilde{u}_1\tilde{u}_2&&strongly~\mbox{ in }~\mathscr{H}~\mbox{ as }~n\rightarrow\infty,\\[1.5mm]
	&\Delta u_i^n(t)\rightharpoonup \Delta\tilde{u}_i(t)\equiv\Delta\tilde{u}_i&&weakly~\mbox{ in }~\mathscr{H}~\mbox{ as }~n\rightarrow\infty,\\[1.5mm]
	&u_i^n(t)\rightarrow \tilde{u}_i(t)\equiv\tilde{u}_i&&strongly~\mbox{ in }~L^2(0,1;L^2(\partial\Omega))~\mbox{ as }~n\rightarrow\infty,\\[1.5mm]
	&|u_2^n(t)|^{\gamma-2}u_2^n(t) \rightharpoonup |\tilde{u}_2|^{\gamma-2}\tilde{u}_2&&weakly~\mbox{ in }~L^2(0,1;L^2(\partial\Omega))~\mbox{ as }~n\rightarrow\infty,\\[1.5mm]
	&\partial_\nu u_i^n(t) \rightharpoonup \partial_\nu \tilde{u}_i&&weakly~\mbox{ in }~L^2(0,1;L^2(\partial\Omega))~\mbox{ as }~n\rightarrow\infty.
	\end{align}
	Thus \(\tilde{u}_1\) and \(\tilde{u}_2\) satisfy
	\begin{equation*}
	\left\{
	\begin{aligned}
	&-\Delta \tilde{u}_1 = \tilde{u}_1\tilde{u}_2 -b\tilde{u}_1,&&x\in\Omega,\\
	&-\Delta \tilde{u}_2 = a\tilde{u}_1,&&x\in\Omega,\\
	&\partial_\nu \tilde{u}_1 + \alpha \tilde{u}_1 = \partial_\nu \tilde{u}_2 + \beta |\tilde{u}_2|^{\gamma-2}\tilde{u}_2=0,&&x\in\partial\Omega.
	\end{aligned}
	\right.
	\end{equation*}
	
	\vspace{3mm}
	(2) Let \(\gamma=2\) and \(\alpha\le2\beta\). By the comparison theorem, we can assume without loss of generality that \(u_{10}(x)=l_1\overline{u}_1(x)\), \(u_{20}(x)=l_2\overline{u}_2(x)\) for some \(l_1>l_2>1\). Suppose that the solution \((u_1,u_2)\) of \(\eqref{NR}\) exists globally, i.e.,
	\begin{equation}
		\label{4.6}
		\sup_{t\in[0,T]}\|u_i(\cdot,t)\|_{L^{\infty}(\Omega)}<\infty,\hspace{5mm} (i=1,2)\hspace{4mm}\forall~ T>0.
	\end{equation}
	Now we are going to construct a subsolution. For this purpose, we first note that there exists a sufficiently small number \(\varepsilon>0\) such that
	\begin{equation}
		\label{4.7}
		\left\{
		\begin{aligned}
			&a(l_2-l_1)\overline{u}_1+\varepsilon l_2\overline{u}_2< 0 &&\mbox{on }\overline{\Omega},\\
			&\varepsilon +(1-l_2)\overline{u}_2< 0 &&\mbox{on }\overline{\Omega}.
		\end{aligned}
		\right.
	\end{equation}
	Here we used the fact that \(\overline{u}_1(x)>0\), \(\overline{u}_2(x)>0\) on \(\overline{\Omega}\), which is assured by Hopf's type maximum principle.
	Let \(u_1^*(x,t)=l_1e^{\varepsilon t}\overline{u}_1(x)\) and \(u_2^*(x,t)=l_2e^{\varepsilon t}\overline{u}_2(x)\). Then using \(\eqref{4.7}\), we get
	\begin{align*}
		\partial_t u_1^* - \Delta u_1^* -u_1^*u_2^* +bu_1^*&= \varepsilon l_1e^{\varepsilon t}\overline{u}_1 -l_1e^{\varepsilon t}\Delta \overline{u}_1 - l_1e^{\varepsilon t}\overline{u}_1 l_2e^{\varepsilon t}\overline{u}_2 + bl_1e^{\varepsilon t}\overline{u}_1\\[1.5mm]
		&=\varepsilon l_1e^{\varepsilon t}\overline{u}_1 + l_1e^{\varepsilon t}\left( \overline{u}_1\overline{u}_2-b\overline{u}_1 \right) - l_1e^{\varepsilon t}\overline{u}_1 l_2e^{\varepsilon t}\overline{u}_2 + bl_1e^{\varepsilon t}\overline{u}_1\\[1.5mm]
		&\le\varepsilon l_1e^{\varepsilon t}\overline{u}_1 + l_1e^{\varepsilon t}\overline{u}_1\overline{u}_2 - l_1l_2e^{\varepsilon t}\overline{u}_1\overline{u}_2\\[1.5mm]
		&=\left\{ \varepsilon + \left( 1-l_2 \right) \overline{u}_2 \right\} l_1e^{\varepsilon t}\overline{u}_1<0,\\[1.5mm]
		\partial_t u_2^* - \Delta u_2^* -au_1^*&=\varepsilon l_2e^{\varepsilon t}\overline{u}_2 - l_2e^{\varepsilon t}\Delta\overline{u}_2 -al_1e^{\varepsilon t} \overline{u}_1\\[1.5mm]
		&=\varepsilon l_2e^{\varepsilon t}\overline{u}_2 + l_2e^{\varepsilon t}a\overline{u}_1 -al_1e^{\varepsilon t} \overline{u}_1\\[1.5mm]
		&=\left\{ \varepsilon l_2\overline{u}_2 + a\left( l_2-l_1 \right) \overline{u}_1 \right\} e^{\varepsilon t}<0,
	\end{align*}
	where we used the fact that \((\overline{u}_1,\overline{u}_2)\) satisfies
	\begin{equation*}
		\left\{
		\begin{aligned}
			&-\Delta \overline{u}_1 = \overline{u}_1\overline{u}_2 -b \overline{u}_1,\\
			&-\Delta \overline{u}_2 = a\overline{u}_1.
		\end{aligned}
		\right.
	\end{equation*}
	Moreover \(\partial_\nu u_1^*+\alpha u_1^*=0\), \(\partial_\nu u_2^*+\beta u_2^*=0\) on \(\partial\Omega\) and \(u_1^*(x,0)=l_1\overline{u}_1(x)\), \(u_2^*(x,0)=l_2\overline{u}_2(x)\). Hence by the comparison principle, we have
	\begin{equation}
		\label{4.8}
		l_1e^{\varepsilon t}\overline{u}_1(x)=u_1^*(x,t)\le u_1(x,t),~~~l_2e^{\varepsilon t}\overline{u}_2(x)=u_2^*(x,t)\le u_2(x,t).
	\end{equation}
	Multiplication of equations of \(\eqref{NR}\) by \(\varphi_1\) and integration by parts yield
	\begin{align}
		\label{4.9}
		&\frac{d}{dt}\left(\int_{\Omega}u_1\varphi_1dx\right) +(b+\lambda_1)\int_{\Omega}u_1\varphi_1dx = \int_{\Omega}u_1u_2\varphi_1dx,\\[1.5mm]
		\label{4.10}
		&\frac{d}{dt}\left(\int_{\Omega}u_2\varphi_1dx\right) + \lambda_1\int_{\Omega}u_2\varphi_1dx +(\beta-\alpha)\int_{\partial\Omega}u_2\varphi_1dS=a\int_{\Omega}u_1\varphi_1dx,
	\end{align}
	where \(\lambda_1\) and \(\varphi_1\) are the first eigenvalue and the corresponding eigenfunction for \(\eqref{eigen}\). We here normalize \(\varphi_1\) so that \(\|\varphi_1\|_{L^1(\Omega)}=1\).
	Substituting \(\eqref{4.10}\) and \(u_1=\frac{1}{a}(\partial_tu_2-\Delta u_2)\) in \(\eqref{4.9}\) and using integration by parts, we get
	\begin{align}
		\label{4.11}&\frac{d}{dt} \left\lbrace \frac{d}{dt}\left(\int_{\Omega}u_2\varphi_1dx\right) + \lambda_1\int_{\Omega}u_2\varphi_1dx +(\beta-\alpha)\int_{\partial\Omega}u_2\varphi_1dS \right\rbrace \\[1.5mm]
		&\nonumber~~+(b+\lambda_1) \left\lbrace \frac{d}{dt}\left(\int_{\Omega}u_2\varphi_1dx\right) + \lambda_1\int_{\Omega}u_2\varphi_1dx +(\beta-\alpha)\int_{\partial\Omega}u_2\varphi_1dS \right\rbrace \\[1.5mm]
		&\nonumber\hspace{-3mm}=\frac{1}{2}\frac{d}{dt}\int_{\Omega}u_2^2\varphi_1dx +\int_{\Omega}|\nabla u_2|^2\varphi_1dx +\frac{\lambda_1}{2}\int_{\Omega}u_2^2\varphi_1dx+ \left(\beta-\frac{\alpha}{2}\right)\int_{\partial\Omega}u_2^2\varphi_1dS,
	\end{align}
	where we used the fact that 
	\begin{align*}
		-\int_{\Omega}(\Delta u_2) u_2\varphi_1dx&=\int_{\Omega}\nabla u_2\cdot\nabla (u_2\varphi_1)dx -\int_{\partial\Omega}(\partial_\nu u_2) u_2\varphi_1dS\\[1.5mm]
		&=\int_{\Omega} |\nabla u_2|^2\varphi_1dx + \int_{\Omega} u_2\nabla u_2\cdot\nabla\varphi_1dx +\beta\int_{\partial\Omega} u_2^2\varphi_1dS\\[1.5mm]
		&=\int_{\Omega} |\nabla u_2|^2\varphi_1dx + \frac{1}{2}\int_{\Omega}\nabla u_2^2\cdot\nabla\varphi_1dx +\beta\int_{\partial\Omega} u_2^2\varphi_1dS\\[1.5mm]
		&=\int_{\Omega} |\nabla u_2|^2\varphi_1dx - \frac{1}{2}\int_{\Omega} u_2^2\Delta\varphi_1dx - \frac{\alpha}{2}\int_{\partial\Omega}u_2^2\varphi_1dS +\beta\int_{\partial\Omega} u_2^2\varphi_1dS\\[1.5mm]
		&=\int_{\Omega} |\nabla u_2|^2\varphi_1dx + \frac{\lambda_1}{2}\int_{\Omega}u_2^2\varphi_1dx + \left(\beta- \frac{\alpha}{2}\right)\int_{\partial\Omega} u_2^2\varphi_1dS.
	\end{align*}
	We here assume \(\beta-\alpha>0\). From \(\eqref{4.8}\), it follows that
	\begin{align*}
		&\frac{\lambda_1}{2}\int_{\Omega}u_2^2\varphi_1dx - (b+\lambda_1)\lambda_1\int_{\Omega}u_2\varphi_1dx\\[1.5mm]
		=&~\frac{\lambda_1}{4}\int_{\Omega}u_2^2\varphi_1dx + \lambda_1 \int_{\Omega}\left\{ \frac{1}{4}u_2-(b+\lambda_1) \right\}u_2\varphi_1dx\\[1.5mm]
		\ge&~ \frac{\lambda_1}{4}\int_{\Omega}u_2^2\varphi_1dx + \lambda_1 \int_{\Omega}\left\{ \frac{1}{4}u_2^*-(b+\lambda_1) \right\}u_2\varphi_1dx\\[1.5mm]
		\ge&~ \frac{\lambda_1}{4}\int_{\Omega}u_2^2\varphi_1dx + \lambda_1 \int_{\Omega}\left\{ \frac{1}{4}me^{\varepsilon t}-(b+\lambda_1) \right\}u_2\varphi_1dx,
	\end{align*}
	where \(m:=\min_{x\in\overline{\Omega}}l_2\overline{u}_2(x)>0\). Hence there exists \(t_1>0\) such that
	\begin{equation}
		\label{4.12}
		\frac{\lambda_1}{2}\int_{\Omega}u_2^2\varphi_1dx - (b+\lambda_1)\lambda_1\int_{\Omega}u_2\varphi_1dx \ge \frac{\lambda_1}{4}\int_{\Omega}u_2^2\varphi_1dx \hspace{10mm}\forall~t\ge t_1.
	\end{equation}
	Similarly, since  
	\begin{align*}
		&\left( \beta-\frac{\alpha}{2} \right)\int_{\partial\Omega}u_2^2\varphi_1dS-(b+\lambda_1)(\beta-\alpha)\int_{\partial\Omega}u_2\varphi_1dS\\[1.5mm]
		=&\frac{1}{2}\left( \beta-\frac{\alpha}{2} \right)\int_{\partial\Omega}u_2^2\varphi_1dS + \int_{\partial\Omega}\left\lbrace \frac{1}{2}\left( \beta-\frac{\alpha}{2} \right)u_2-(b+\lambda_1)(\beta-\alpha) \right\rbrace u_2\varphi_1dS\\[1.5mm]
		\ge& \frac{1}{2}\left( \beta-\frac{\alpha}{2} \right)\int_{\partial\Omega}u_2^2\varphi_1dS + \int_{\partial\Omega}\left\{ \frac{1}{2}\left( \beta-\frac{\alpha}{2} \right)me^{\varepsilon t}  - (b+\lambda_1)(\beta-\alpha) \right\}u_2\varphi_1dS,
	\end{align*}
	there exists \(t_2>0\) such that 
	\begin{align}
		\nonumber\left( \beta-\frac{\alpha}{2} \right)\int_{\partial\Omega}u_2^2\varphi_1dS-(b+\lambda_1)&(\beta-\alpha)\int_{\partial\Omega}u_2\varphi_1dS\\[1.5mm]
		&\ge~\frac{1}{2}\left( \beta-\frac{\alpha}{2} \right)\int_{\partial\Omega}u_2^2\varphi_1dS\hspace{10mm}\forall~t\ge t_2.\label{4.13}
	\end{align}
	Therefore by \(\eqref{4.12}\), \(\eqref{4.13}\) and \(\eqref{4.11}\), we have
	\begin{align}
		\nonumber
		&\frac{d}{dt}\left\{ \frac{d}{dt}\left(\int_{\Omega}u_2\varphi_1dx\right) \right\} + (b+2\lambda_1)\frac{d}{dt}\left(\int_{\Omega}u_2\varphi_1dx\right) + (\beta-\alpha)\frac{d}{dt}\left(\int_{\partial\Omega}u_2\varphi_1 dS\right)\\[1.5mm]
		\ge&\frac{1}{2}\frac{d}{dt}\left(\int_{\Omega}u_2^2\varphi_1dx\right) + \frac{\lambda_1}{4}\int_{\Omega}u_2^2\varphi_1dx+\frac{1}{2}\left( \beta-\frac{\alpha}{2} \right)\int_{\partial\Omega}u_2^2\varphi_1dS\hspace{10mm} \forall~t\ge t_3,\label{4.14}
	\end{align}
	where \(t_3:=t_1\vee t_2\).
	Now we integrate \(\eqref{4.14}\) with respect to \(t\) over \([t_3,t]\) to get
	\begin{align}
		\nonumber
		&\frac{d}{dt}\left\{ \int_{\Omega}u_2\varphi_1dx+(\beta-\alpha)\int_{t_3}^{t}\int_{\partial\Omega}u_2\varphi_1dSd\tau \right\}\\[1mm]
		\ge&~\frac{1}{2}\int_{\Omega}u_2^2\varphi_1dx - (b+2\lambda_1)\int_{\Omega}u_2\varphi_1dx -\frac{1}{2}\int_{\Omega}u_2^2(t_3)\varphi_1dx\nonumber\\[1mm]
		&+ \frac{1}{2}\left( \beta-\frac{\alpha}{2} \right)\int_{t_3}^{t}\int_{\partial\Omega}u_2^2\varphi_1dSd\tau + \int_{\Omega}\partial_t u_2(t_3)\varphi_1dx,\label{4.15}
	\end{align}
	where we neglected positive terms. Moreover we can see that there exists \(t_4>t_3\) such that
	\begin{align}
		\nonumber
		\frac{1}{2}\int_{\Omega}u_2^2\varphi_1dx - &(b+2\lambda_1)\int_{\Omega}u_2\varphi_1dx  \\[1mm]
		&-\frac{1}{2}\int_{\Omega}u_2^2(t_3)\varphi_1dx+ \int_{\Omega}\partial_t u_2(t_3)\varphi_1dx\ge \frac{1}{4}\int_{\Omega}u_2^2\varphi_1dx\label{4.16}
	\end{align}
	for \(t\ge t_4\) by the same argument as before. Therefore from \(\eqref{4.15}\) and \(\eqref{4.16}\), we have
	\begin{align}
		\nonumber
		\frac{d}{dt}&\left\{ \int_{\Omega}u_2\varphi_1dx + (\beta-\alpha)\int_{t_3}^{t}\int_{\partial\Omega}u_2\varphi_1dSd\tau \right\}\\[1mm]
		&\phantom{aaaaaaaaaaaa}\ge~\frac{1}{4}\int_{\Omega}u_2^2\varphi_1dx	+ \frac{1}{2}\left( \beta-\frac{\alpha}{2} \right)\int_{t_3}^{t}\int_{\partial\Omega}u_2^2\varphi_1dSd\tau.\label{4.17}
	\end{align}
	Since \(\|\varphi_1\|_{L^1(\Omega)}=1\), by Schwarz's inequality, we get
	\begin{equation*}
		\frac{1}{4}\int_{\Omega}u_2^2\varphi_1dx\ge\frac{1}{4}\left( \int_{\Omega}u_2\varphi_1dx \right)^2,
	\end{equation*}
	and
	\begin{align*}
		&\frac{1}{2}\left( \beta-\frac{\alpha}{2} \right)\int_{t_3}^{t}\int_{\partial\Omega}u_2^2\varphi_1dSd\tau\\[1.5mm]
		\ge&~\frac{1}{2}(\beta-\frac{\alpha}{2})\frac{1}{\|\varphi_1\|_{L^{\infty}(\Omega)}|\partial\Omega|}\frac{1}{t-t_3}\left\{ \int_{t_3}^{t}\int_{\partial\Omega}u_2\varphi_1dSd\tau \right\}^2\\[1.5mm]
		=&~\frac{1}{2}\frac{\beta-\frac{\alpha}{2}}{\|\varphi_1\|_{L^{\infty}(\Omega)}|\partial\Omega|(\beta-\alpha)^2}\frac{1}{t-t_3}\left\{(\beta-\alpha) \int_{t_3}^{t}\int_{\partial\Omega}u_2\varphi_1dSd\tau \right\}^2.
	\end{align*}
	By the above inequalities and \(\eqref{4.17}\), for \(t\ge t_5:=t_4\vee(t_3+1)\), we finally get 
	\begin{align*}
		&\frac{d}{dt}\left\{ \int_{\Omega}u_2\varphi_1dx+(\beta-\alpha)\int_{t_3}^{t}\int_{\partial\Omega}u_2\varphi_1dSd\tau \right\}\\[1.5mm]
		\ge&\frac{1}{4}\int_{\Omega}u_2^2\varphi_1dx + \frac{1}{2}\left( \beta-\frac{\alpha}{2} \right)\int_{t_3}^{t}\int_{\partial\Omega}u_2^2\varphi_1dSd\tau\\[1.5mm]
		\ge&\frac{1}{4}\left( \int_{\Omega}u_2\varphi_1dx \right)^2 + \frac{1}{2}\frac{\beta-\frac{\alpha}{2}}{\|\varphi_1\|_{L^{\infty}(\Omega)}|\partial\Omega|(\beta-\alpha)^2}\frac{1}{t-t_3}\left\{(\beta-\alpha) \int_{t_3}^{t}\int_{\partial\Omega}u_2\varphi_1dSd\tau \right\}^2\\[1.5mm]
		\ge&C\frac{1}{t-t_3} \left\{ \left( \int_{\Omega}u_2\varphi_1dx \right)^2 + \left((\beta-\alpha) \int_{t_3}^{t}\int_{\partial\Omega}u_2\varphi_1dSd\tau \right)^2 \right\}\\[1.5mm]
		\ge&C\frac{1}{t-t_3} \left\{  \int_{\Omega}u_2\varphi_1dx  + (\beta-\alpha) \int_{t_3}^{t}\int_{\partial\Omega}u_2\varphi_1dSd\tau  \right\}^2,
	\end{align*}
	where \(C\) denotes some general positive constant independent of \(t\).
	Set \(y(t):=\int_{\Omega}u_2\varphi_1dx + (\beta-\alpha)\int_{t_3}^{t}\int_{\partial\Omega}u_2\varphi_1dSd\tau\), then the above inequality yields the following:
	\begin{equation*}
		\left\{
		\begin{aligned}
			&\frac{d}{dt}y(t)\ge\frac{C}{t-t_3} y^2(t)&&t\ge t_5,\\
			&y(t_5)>0.
		\end{aligned}
		\right.
	\end{equation*}
	We can see that there exists \(T^*>t_5\) such that
	\begin{equation}
		\label{blow}
		\lim_{t\rightarrow T^*}y(t)=+\infty.
	\end{equation}
	In order to show the existence of \(T^*\) satisfying \(\eqref{blow}\), it suffices to consider the following ordinary differential equation:
	\begin{equation*}
		\left\{
		\begin{aligned}
			&\frac{d}{dt}\tilde{y}(t) = \frac{C}{t-t_3} \tilde{y}^2(t)&&t\ge t_5,\\
			&\tilde{y}(t_5)>0.
		\end{aligned}
		\right.
	\end{equation*}
	Since \(\frac{d}{dt}\tilde{y}(t)>0\) for all \(t\ge t_5\) and \(\tilde{y}(t_5)>0\), it is clear that \(\tilde{y}(t)>0\) for all \(t\ge t_5\). Divide both sides by \(\tilde{y}^2(t)\) and integrate with respect to \(t\) on \([t_5,t]\), then we have
	\begin{equation*}
		\frac{1}{\tilde{y}^2(t)}\frac{d}{dt} \tilde{y}(t) = \frac{C}{t-t_3},
	\end{equation*} 
	\begin{equation*}
		\int_{\tilde{y}(t_5)}^{\tilde{y}(t)} \frac{1}{y^2} dy = C \log\frac{t-t_3}{t_5-t_3},
	\end{equation*}
	\begin{equation*}
		-\frac{1}{\tilde{y}(t)} + \frac{1}{\tilde{y}(t_5)} = C \log\frac{t-t_3}{t_5-t_3}.
	\end{equation*}
	Therefore we have 
	\begin{equation*}
		\tilde{y}(t)=\frac{1}{\frac{1}{\tilde{y}(t_5)} - C \log\frac{t-t_3}{t_5-t_3}}.
	\end{equation*}
	Hence there exists \(\tilde{T}>t_5\) satisfying 
	\begin{equation*}
		\frac{1}{\tilde{y}(t_5)} - C \log\frac{\tilde{T}-t_3}{t_5-t_3}=0
	\end{equation*}
	such that 
	\begin{equation*}
		\lim_{t\rightarrow \tilde{T}}\tilde{y}(t)=+\infty.
	\end{equation*}
	Thus \(\eqref{blow}\) holds by comparison theorem for ordinary differential equations.
	This contradicts the assumption that \((u_1,u_2)\) exists globally. \\

	For the case of \(\frac{\alpha}{2}\le \beta \le \alpha\), we can prove the same result with a slight modification. Actually, we get from \(\eqref{4.11}\)
	\begin{align*}
		&\frac{d}{dt} \left\lbrace \frac{d}{dt}\left(\int_{\Omega}u_2\varphi_1dx\right) + \lambda_1\int_{\Omega}u_2\varphi_1dx +(\beta-\alpha)\int_{\partial\Omega}u_2\varphi_1dS \right\rbrace \\[1.5mm]
		&\nonumber~~+(b+\lambda_1) \left\lbrace \frac{d}{dt}\left(\int_{\Omega}u_2\varphi_1dx\right) + \lambda_1\int_{\Omega}u_2\varphi_1dx \right\rbrace \\[1.5mm]
		&\nonumber\hspace{-3mm}\ge\frac{1}{2}\frac{d}{dt}\int_{\Omega}u_2^2\varphi_1dx +\frac{\lambda_1}{2}\int_{\Omega}u_2^2\varphi_1dx.
	\end{align*}
	Using \(\eqref{4.12}\) and integrating above inequality with respect to \(t\) over \([t_1,t]\), we have
	\begin{align*}
		&\frac{d}{dt}\left(\int_{\Omega}u_2\varphi_1dx \right)- \int_{\Omega}\partial_t u_2(t_1)\varphi_1dx + (\beta-\alpha)\int_{\partial\Omega}u_2\varphi_1dS -  (\beta-\alpha)\int_{\partial\Omega}u_2(t_1)\varphi_1dS\\[1.5mm]
		\ge&\frac{1}{2}\int_{\Omega}u_2^2\varphi_1dx-\frac{1}{2}\int_{\Omega}u_2^2(t_3)\varphi_1dx - (b+2\lambda_1)\int_{\Omega}u_2\varphi_1dx + (b+2\lambda_1)\int_{\Omega}u_2(t_1)\varphi_1dx .
	\end{align*}
	Repeating the same arguments as for \(\eqref{4.12}\), we see that there exists \(t_6\ge t_1\) such that
	\begin{align*}
		&\frac{1}{2}\int_{\Omega}u_2^2\varphi_1dx-\frac{1}{2}\int_{\Omega}u_2^2(t_3)\varphi_1dx - (b+2\lambda_1)\int_{\Omega}u_2\varphi_1dx \\[1.5mm]
		&+ \int_{\Omega}\partial_tu_2(t_1)\varphi_1dx + (\beta-\alpha)\int_{\partial\Omega}u_2(t_1)\varphi_1dS\\[1.5mm]
		\ge~&\frac{1}{4}\int_{\Omega}u_2^2\varphi_1dx
	\end{align*}
	for all \(t\ge t_6\). From these inequalities and Schwarz's inequality, it holds that
	\begin{equation*}
		\frac{d}{dt}\left(\int_{\Omega}u_2\varphi_1dx \right) \ge \frac{1}{4}\left( \int_{\Omega}u_2\varphi_1dx \right)^2\hspace{5mm}\forall ~t\ge t_6.
	\end{equation*} 
	Therefore we can get the following differential inequality:
	\begin{equation*}
		\left\{
		\begin{aligned}
			&\frac{d}{dt} y(t) \ge y^2(t) &&\hspace{5mm}t\ge t_6,\\
			&y(t_6)>0,
		\end{aligned}
		\right.
	\end{equation*}
	where \( y(t)=\int_{\Omega}u_2\varphi_1dx\). It is easy to see that there exists \(T^{**}>t_6\) such that 
	\begin{equation*}
		\lim_{t\rightarrow T^{**}}y(t)=+\infty.
	\end{equation*}  
	This leads to a contradiction.
\end{proof}

\begin{remark}\upshape
	Since the blow-up result is proved by contradiction, there is no knowing if \(\|u_1(t)\|_{L^{\infty}}\) and \(\|u_2(t)\|_{L^{\infty}}\) blow up simultaneously.
	However we can show by another argument that \(L^{\infty}\)-norms of \(u_1\) and \(u_2\) blow up at the same time, i.e., there exists \(T>0\) such that 
	\begin{equation*}
		\lim_{t\rightarrow T}\|u_1(t)\|_{L^{\infty}(\Omega)} = \infty~~~\mbox{and}~~~\lim_{t\rightarrow T}\|u_2(t)\|_{L^{\infty}(\Omega)} = \infty.
	\end{equation*}
	In fact, multiplying the first equation of \(\eqref{NR}\) by \(|u_1|^{r-2}u_1\) and using integration by parts and similar calculation in the proof of Theorem \ref{LWP}, we obtain
	\begin{equation}
	\label{same1}
	\frac{d}{dt}\|u_1(t)\|_{L^r(\Omega)}\le \|u_2(t)\|_{L^{\infty}(\Omega)} \|u_1(t)\|_{L^r(\Omega)} \hspace{10mm}\forall t\in[0,T).
	\end{equation}
	From the second equation of \(\eqref{NR}\), we also have 
	\begin{equation}
	\label{same2}
	\|u_2(t)\|_{L^{\infty}(\Omega)} \le \|u_{20}\|_{L^{\infty}(\Omega)} + a \int_{0}^{t}\|u_1(\tau)\|_{L^{\infty}(\Omega)}d\tau \hspace{10mm}\forall t\in[0,T).
	\end{equation}
	Suppose that
	\begin{equation*}
	\lim_{t\rightarrow T} \|u_1(t)\|_{L^{\infty}(\Omega)} = \infty\hspace{7mm}\mbox{and}\hspace{7mm}M_2:=\sup_{0\le t\le T}\|u_2(t)\|_{L^{\infty}(\Omega)}<\infty,
	\end{equation*}
	then it follows from \(\eqref{same1}\)
	\begin{equation*}
	\frac{d}{dt} \|u_1(t)\|_{L^r}\le M_2 \|u_1(t)\|_{L^r(\Omega)} \hspace{10mm}\forall t\in[0,T). 
	\end{equation*}
	By Gronwall's inequality, we get
	\begin{equation*}
	\|u_1(t)\|_{L^r(\Omega)} \le \|u_{10}\|_{L^r(\Omega)} e^{M_2t} \le \|u_{10}\|_{L^r(\Omega)} e^{M_2T} \hspace{10mm}\forall t\in[0,T).
	\end{equation*}
	Letting \(r\) tend to \(\infty\), we obtain
	\begin{equation*}
	\|u_1(t)\|_{L^{\infty}(\Omega)} \le \|u_{10}\|_{L^{\infty}(\Omega)} e^{M_2T} \hspace{10mm}\forall t\in[0,T),
	\end{equation*}
	which contradicts the fact \(\lim_{t\rightarrow T}\|u_1(t)\|_{L^{\infty}(\Omega)}=\infty\). Next, suppose that
	\begin{equation*}
	M_1:=\sup_{0\le t\le T}\|u_1(t)\|_{L^{\infty}(\Omega)}<\infty\hspace{7mm}\mbox{and}\hspace{7mm}\lim_{t\rightarrow T} \|u_2(t)\|_{L^{\infty}(\Omega)} = \infty,
	\end{equation*}
	then by \(\eqref{same2}\) we see that
	\begin{align*}
	\|u_2(t)\|_{L^{\infty}(\Omega)}	&\le \|u_{20}\|_{L^{\infty}(\Omega)} + aM_1T \hspace{10mm}\forall t\in[0,T).
	\end{align*}
	Letting \(t\) tend to \(T\), we get contradiction. Thus we see that \(u_1\) and \(u_2\) blow up at the same time.

\end{remark}

\end{document}